\documentclass{amsart}
\usepackage{graphicx}
\usepackage[english]{babel}
\usepackage[latin1]{inputenc}
\usepackage[all, cmtip]{xy}
\usepackage{amsthm,amsmath,amsfonts,amscd,amssymb}
\numberwithin{equation}{section}

\newtheorem{thm}{Theorem}[section]
\newtheorem{lem}[thm]{Lemma}
\newtheorem{prop}[thm]{Proposition}

\newtheorem{rem}[thm]{Remark}
\newtheorem{exam}[thm]{Example}

\newtheorem{dfn}[thm]{Definition}

\newcommand{\Arrow}{\rightarrow}
\newcommand{\lrrow}{\longrightarrow}

\newcommand{\Cour}[1]      {[\![#1]\!]}

\newcommand{\p}            {^{\perp}}

\newcommand{\Ann}[1]       {\mathrm{Ann}\left(#1\right)}
\newcommand{\End}[1]       {\mathrm{End}\left(#1\right)}

\newcommand{\Ext}[1]       {\wedge^{\bullet} #1}
\newcommand{\Graph}[1]     {\mathrm{Graph}\left(#1\right)}

\newcommand{\pr}           {\mathrm{pr}}
\newcommand{\Aut}          {\mathrm{Aut}}
\newcommand{\Der}          {{\mathrm {Der}}}

\newcommand{\op}[1]        {#1^{\vphantom{}^{\mathrm{op}}}}

\newcommand{\frakg}     {\mathfrak{g}}
\newcommand{\Lie}       {\mathcal{L}}

\newcommand{\frakh}     {\mathfrak{h}}

\newcommand{\ad}        {{\mathrm {ad}}}

\newcommand{\A}         {\mathcal{A}}

\newcommand{\Diff}       {{\mathrm {Diff}}}

\newcommand{\C}            {\mathbb{C}}
\newcommand{\R}            {\mathbb{R}}

\newcommand{\T}            {\mathbb{T}}
\newcommand{\J}            {\mathcal{J}}

\begin{document}
\title{Generalized reduction and pure spinors.}
\author{T. Drummond }
\address{Universidade Federal do Rio de Janeiro, Instituto de Matem\'atica, 21945-970, Rio de Janeiro - Brazil.}
\email{drummond@im.ufrj.br}
\begin{abstract}
We study reduction of Dirac structures as developed by Bursztyn, Cavalcanti and Gualtieri \cite{b_c_g} from the point of view of pure spinors. We \linebreak describe explicitly the pure spinor line bundle of the reduced Dirac structure. We also obtain results on reduction of generalized Calabi-Yau structures.
\end{abstract}
\maketitle

\tableofcontents

\section*{Introduction.}

In this paper we study symmetries and reduction of Dirac structures \cite{courant, liu_w_xu} and generalized complex structures \cite{gualt_thesis, hitchin} from the viewpoint of pure spinors. As discussed in \cite{gualt_thesis}, pure spinors provide a key tool in generalized complex geometry, specially to study  generalized Calabi Yau structures. Our main focus is the interplay between Dirac structures/generalized complex structures and pure spinors in the reduction procedure introduced by H. Bursztyn, G. Cavalcanti and M. Gualtieri \cite{b_c_g} (see also \cite{lin_tolman, stie_xu} for other reduction procedures). More precisely, the  main goal of this work is to find an explicit  description of the pure spinor line bundle associated to the Dirac structure $L_{red}$ obtained by reducing a Dirac structure $L$ under the procedure of \cite{b_c_g}. In \cite{nitta}, Y. Nitta introduced a procedure to reduce pure spinors in the context of generalized Calabi-Yau structures. The motivation for this paper is to put the work of Y. Nitta into the broader perspective of \cite{b_c_g} and to move toward a general procedure to reduce generalized Calabi Yau structures.

The set-up for the reduction on \cite{b_c_g} is given by an action of a (compact, connected) Lie group $G$ on $M$ which is lifted to an action on Courant algebroid $E$ by automorphisms and an invariant submanifold $N \subset M$ on which $G$ acts  freely. With these data in hands, one constructs an isotropic subbundle $K \subset E|_N$ from which one obtains a reduced Courant algebroid $E_{red}$ over $M_{red}= N/G$ in which the reduced Dirac structures will be defined (see Theorem \ref{reduction_courant}). The reduction procedure itself start with an invariant Dirac structure $L \subset E$ for which $L|_N \cap K$ has constant rank (a clean intersection condition) and gives the reduced Dirac structure $L_{red} \subset E_{red}$. Our approach to find the reduced pure spinor line bundle is based on a new description of $L_{red}$ (see Proposition \ref{L_char}) which, as an aside, relates the procedure of M. Stienon and P. Xu \cite{stie_xu} to reduce generalize complex structures with that of \cite{b_c_g}. We prove that, up to the choice of a connection on $N$, $L_{red}$ is obtained from $L$ by a combination of pull-back by the inclusion map $j: N \Arrow M$ and push-forward by the quotient map $q: N \Arrow M_{red}$ (see \cite{bursztyn_radko} for the definition of such operations) exactly as in \cite{stie_xu}. The choice of a connection has the cost of introducing a 2-form  $B\in \Omega^2(N)$ (which should be thought of as a change of coordinates) in the picture. From this, it is a simple matter to describe the pure spinor line bundle of $L_{red}$: if $\varphi \in \Gamma(\Ext{T^*M})$ is a pure spinor for $L$, then
\begin{equation}\label{intro_eq}
\varphi_{red} = q_* (e^B \wedge j^*\varphi)
\end{equation}
is a pure spinor for $L_{red}$, where $q_*: \Omega(N) \Arrow \Omega(M_{red})$ is the push-forward of differential forms on the principal bundle $q:N \Arrow M_{red}$ (see Theorem \ref{main_transv}). Moreover, we prove that, for $x \in  N$, $\varphi_{red}|_{q(x)} \neq 0$ if and only if $L_x \cap K_x = 0$ (a transversality condition). In \S 4.1, we use formula \eqref{intro_eq} to give an alternative explanation of a recent result of G. Cavalcanti and M. Gualtieri \cite{c_g} relating T-duality \cite{bou_ev, bou_hanna} to generalized geometry. 

In important examples (e.g.  Nitta's reduction procedure), the transversality condition $L|_N\cap K=0$ is not satisfied and, therefore, formula \eqref{intro_eq} does not solve the problem of describing the pure spinor line bundle of $L_{red}$. Overcoming this transversality issue is the main technical problem we have to handle. To do that, we develop a pertubation procedure that replaces $L$ by a new Lagrangian subbundle $L_D$ of $E|_N$ (depending on the choice of a subbundle $D\subset E|_{N}$) satisfying the transversality condition $L_D \cap K=0$ and producing the same reduced Dirac structure $L_{red}$ as $L$ (see Proposition \ref{pertubation}). At the pure spinor level, the change from $L$ to $L_D$ has the effect of changing the pure spinor $\varphi$ to $\varphi_D$, which is obtained from $\varphi$ by Clifford multiplication by an element of $\det(D)$. In Theorem \ref{main_geral}, we use this pertubation procedure to obtain the general form of a nowhere-zero section of the reduced pure spinor line bundle corresponding to $L_{red}$:
\begin{equation}\label{intro_eq2}
\varphi_{red} = q_*(e^B \wedge j^*\varphi_D).
\end{equation}
This pertubation method enables us  to recover Nitta's reduction \cite{nitta} as a particular instance of \eqref{intro_eq2} when $D$ is properly choosen. We are also able to obtain a new result on reduction of generalized Calabi-Yau structures in the presence of transversality conditions (see Proposition \ref{transv_calabi}). The main obstacle for applying our  method to reduce generalized Calabi-Yau structures in a general setting is the lack of a fine control on how the de Rham differential relates with the pertubation $\varphi \mapsto \varphi_D$. In any case, formula \eqref{intro_eq2} gives us the tool to understand how the Chern classes of the pure spinor line bundles of $L$ and $L_{red}$ are related.

The paper is organized as follows. In Section 1, we recall the main definitions of generalized geometry and set notation. In Section 2, we present the BCG reduction procedure and prove Proposition \ref{L_char} which gives a new characterization of the reduced Dirac structure $L_{red}$. Section 3 is devoted to Theorem \ref{main_geral}, our main theorem which gives an explicit description of the pure spinor line bundle of $L_{red}$. In \S 4.1, we move towards a general theorem about reduction of generalized Calabi-Yau structures and show how the work of Y.Nitta \cite{nitta} fits into the framework developed in \S 3. In the last subsection \S 4.2, we explain how Theorem \ref{main_transv} can be applied to recover a result of G. Cavalcanti and M. Gualtieri \cite{c_g} on T-duality. More specifically, we show that the isomorphism of the twisted-cohomologies of T-duals spaces introduced in \cite{bou_ev} is exactly an instance of formula \eqref{intro_eq} when applied to pure spinors. 

\paragraph{Acknowledgements.}
This work is part of the author's PhD thesis, supported by CNPq. Thanks are due specially to Henrique Bursztyn for his guidance. The author would like to thank also C. Ortiz, A. Cabrera and G. Cavalcanti for helpful discussions.

\paragraph{NOTATION.}
In this paper (specially in Chapter 3), given a linear homomorphism $f: V \Arrow W$, we let $f$ denote also the natural extension $f: \Ext{V} \Arrow \Ext{W}$ to exterior algebras. We believe this will cause no confusion.

\section{Preliminaries.}

\subsection{Courant algebroids.}
Let $M$ be a smooth manifold and let $\pi:E \Arrow M$ be a Courant algebroid \cite{liu_w_xu} over $M$ with anchor $p: E \Arrow TM$, fibrewise non-degenerate symmetric bilinear form $\langle \cdot, \cdot \rangle$ and bilinear bracket $\Cour{ \,, }$ on $\Gamma(E)$. We will only deal with exact Courant algebroids.

\begin{dfn}{\em
A Courant algebroid $E$ is said to be \textit{exact} if the sequence
\begin{equation}\label{exact}
0 \lrrow T^*M \stackrel{p^*} \lrrow E \stackrel{p}\lrrow TM \lrrow 0
\end{equation}
is exact (we use $\langle \cdot, \cdot \rangle$ to identify $E \cong E^*$).
}
\end{dfn}

The standard example of an exact Courant algebroid is $\T M :=TM\oplus T^*M$ with $\pr_{TM}: \T M \Arrow TM$ as anchor, the bilinear symmetric form $\langle \cdot, \cdot \rangle$ canonically given by
\begin{equation}\label{can_bin}
\langle X+\xi, Y+\eta \rangle = i_X \eta + i_Y \xi, \text{ for } X+ \xi, Y+ \eta \in \Gamma(\T M)
\end{equation}
and the bracket $\Cour{\cdot, \cdot}$ given by
\begin{equation}\label{Courant_bracket}
\Cour{X+\xi, Y+ \eta} = [X,Y] + \Lie_X \eta -i_Yd\eta.
\end{equation}

Given an exact Courant algebroid $E$, it is always possible to find an \textit{isotropic splitting} $\nabla: TM \Arrow E$ for \eqref{exact}, i.e. a splitting $\nabla$ whose image is isotropic with respect to $\langle \cdot, \cdot \rangle$. The space of isotropic splittings for $E$ is affine over $\Omega^2(M)$, i.e. given a 2-form $B$ and an isotropic splitting $\nabla$, one has that
$$
(\nabla + B) (X) = \nabla X + p^*(i_XB)
$$
is also an isotropic splitting and any two isotropic splitting are in the same $\Omega^2(M)$-orbit. The \textit{curvature} of $\nabla$ is the closed 3-form $H \in \Omega^3(M)$ defined by
$$
H(X,Y,Z) = \langle \Cour{\nabla X, \nabla Y}, Z\rangle.
$$
If we change $\nabla$ to $\nabla+B$, its curvature changes to $H+dB$. The cohomology class $[H] \in H^3(M, \R)$ does not depend on the splitting, and it is called the \textit{\v{S}evera class} of $E$ (see \cite{severa_weinstein} for more details).

For an isotropic splitting $\nabla$ with curvature $H \in \Omega^3(M)$, the isomorphism $\nabla + p^*: TM \oplus T^*M \lrrow E$ identifies the bilinear form and the bracket on $E$ with the bilinear form \eqref{can_bin} and the \textit{$H$-twisted Courant bracket} \cite{severa_weinstein}
\begin{equation}\label{H_twisted}
\Cour{X+\xi, Y+\eta}_H = [X,Y] + \Lie_X \eta - i_Y(d\xi - i_XH)
\end{equation}
respectively, where $X, Y \in \Gamma(TM)$, $\xi, \eta \in \Gamma(T^*M)$.

Define $\Phi_{\nabla}: E \Arrow TM\oplus T^*M$ by
\begin{equation}\label{fi_nabla}
\Phi_{\nabla}= (\nabla + p^*)^{-1}.
\end{equation}
For a 2-form $B \in \Omega^2(M)$, we have
\begin{equation}\label{change_split}
\tau_B \circ \Phi_{\nabla+B} = \Phi_{\nabla},
\end{equation}
where $\tau_B: \T M \Arrow \T M$ is the so called \textit{$B$-field transformation} and it is given by
\begin{equation}\label{B_field}
\tau_B(X+\xi)= X + \xi+i_XB , \,\text{for } X \in TM, \, \xi \in T^*M.
\end{equation}

Let $L \subset E$ be a subbundle. Define
$$
L\p = \{ e \in E \,\, | \,\, \langle e, \cdot \rangle|_L \equiv 0 \}.
$$
$L$ is said to be \textit{isotropic} if $L \subset L\p$ and \textit{Lagragian} if $L=L\p$. Equivalently,  $L$ is Lagrangian if  it is isotropic and $\dim(L)=\dim(M)$. 

\begin{dfn}\cite{courant} {\em
A Lagrangian subbundle $L \subset E$ is said to be \textit{integrable} if $\Cour{\Gamma(L), \Gamma(L)} \subset \Gamma(L)$. In this case we say that $L$ is a \textit{Dirac structure}.
}
\end{dfn}

In this work, we are mainly concerned with Dirac structures. Recently, there has been a lot of interest in these structures motivated by the work of M. Gualtieri in generalized complex geometry and its applications in physics (see \cite{gualt_thesis} and references therein). 

Fix a closed 3-form $H \in \Omega^3(M)$. We finish this subsection giving some examples of Dirac structures $L \subset (\T M , \Cour{\cdot, \cdot}_H)$.

\begin{exam}\label{2_form}{\em
For a 2-form $\omega \in \Omega^2(M)$, its graph $L_{\omega}= \{(X, \omega(X, \cdot)) \,\,| \,\, X \in TM \} \subset \T M$ is a Lagrangian subbundle of $\T M$ and $L_{\omega}$ is integrable if and only if $d\omega = -H$. 
}
\end{exam}

\begin{exam}\label{distributions}{\em
For a distribution $\Delta \subset TM$ on $M$, $L_{\Delta} = \Delta \oplus \Ann{\Delta} \subset \T M$ is a Lagrangian subbundle of $\T M$. $L_{\Delta}$ is integrable if and only if $\Delta= T \mathcal{F} \text{ and } H|_{\mathcal{F}}\equiv 0$ for $\mathcal{F}$ a foliation of $M$.
}
\end{exam}

\subsection{Symmetries.} 
Let $E$ be an exact Courant algebroid over $M$.

\begin{dfn}\label{sym_def}\cite{b_c_g, hu} {\em
The automorphism group $\Aut(E)$ of a Courant algebroid $E$ is the group of pairs $(\Psi, \psi)$, where $\Psi:E\Arrow E$ is a bundle automorphism covering $\psi \in \Diff(M)$ such that
\begin{itemize}
\item[(1)] $\psi^* \langle \Psi( \cdot), \Psi(\cdot)\rangle = \langle \cdot, \cdot\rangle$
\item[(2)] $\Cour{\Psi(\cdot), \Psi(\cdot) } = \Psi\Cour{\cdot, \cdot }$;
\item[(3)] $p \circ \Psi = \psi_* \circ p$.
\end{itemize}
}
\end{dfn}

Fix an isotropic splitting $\nabla : TM \Arrow E$.

\begin{prop}[\cite{b_c_g}]\label{nabla_symmetries}
For $(\Psi, \psi) \in \Aut(E)$, there exists $B \in \Omega^2(M)$ such that
\begin{equation}\label{Phi_B}
\Phi_{\nabla} \circ \Psi \circ \Phi_{\nabla}^{-1} = 
\left(\begin{array}{cc}
\psi_* & 0 \\
0 & (\psi^{-1})^*
\end{array} \right )
\circ 
\tau_B
\end{equation}
where $\tau_B$ is the associated $B$-field transformation \eqref{B_field}. Moreover, if $H \in \Omega^3(M)$ is the curvature of $\nabla$,
\begin{equation}\label{coclass_preserved}
H-\psi^*H =dB.
\end{equation}
\end{prop}

\begin{rem}\label{pertub_split}{\em
For $A=(\Psi, \psi) \in \Aut(E)$, the composition
\begin{equation}\label{nabla_A}
T_x M \stackrel{\psi_*}\lrrow T_{\psi(x)} M \stackrel{\nabla} \lrrow E_{\psi(x)} \stackrel{\Psi^{-1}} \lrrow E_{x}, \, \, x\in M,
\end{equation}
defines an isotropic splitting $\nabla^A$ which differs from $\nabla$ exactly by the 2-form $B$ whose existence is stated in Proposition \ref{nabla_symmetries} (i.e. $\nabla = \nabla^A + B$).
}
\end{rem}

We say that $(\Psi, \psi)$ \textit{preserves the splitting} if the splitting $\nabla^A$ defined by \eqref{nabla_A} equals $\nabla$ (i.e. $\Psi \circ \nabla = \nabla \circ \psi_*$). In this case, 
$$
\Phi_{\nabla} \circ \Psi \circ \Phi_{\nabla}^{-1} = 
\left(\begin{array}{cc}
\psi_* & 0 \\
0 & (\psi^{-1})^*
\end{array} \right )
$$
and $\psi^*H= H$. 

The Lie algebra of derivations $\Der(E)$ is the Lie algebra of covariant differential operators $A: \Gamma(E) \Arrow \Gamma(E)$ covering vector fields $X\in \Gamma(TM)$ such that
$$
\Lie_X \langle \cdot, \cdot \rangle = \langle A \cdot, \cdot\rangle + \langle \cdot, A \cdot\rangle \,\text{ and } \, A\Cour{\cdot, \cdot} = \Cour{A\cdot, \cdot} + \Cour{\cdot, A \cdot}.
$$
Choosing an isotropic splitting $\nabla$, Proposition \ref{nabla_symmetries} identifies $\Der(E)$ with the set of pairs $(X,B) \in \Gamma(TM) \times \Omega^2(M)$ such that $\Lie_X H = -dB$ (where $H$ is the curvature of $\nabla$). The action of $(X,B)$ on $\Gamma(\T M)$ is given by
$$
(X,B)\cdot (Y+\eta) = [X,Y] + \Lie_X \eta - i_Y B.
$$
In particular, given a section $X+\xi$ of $\T M$, one has
$$
(X, d\xi - i_XH) \cdot (Y+\eta) = \Cour{X+\xi, Y+\eta}_H.
$$
Hence, $\Cour{X+\xi, \cdot}$ is a derivation of $E$ which we call \textit{inner derivation}. In general, the adjoint map
$$
ad: e \in E \longmapsto \Cour{e, \cdot} \in \Der(E)
$$
is not injective nor surjective  (see Propoisition 2.5. of \cite{b_c_g} for more details).

\subsection{Pure Spinors.}

Let $E$ be an exact Courant algebroid over $M$ and consider the Clifford bundle $Cl(E)$ associated to $(E, \langle \cdot, \cdot \rangle)$ (see \cite{spin_geo}). The fiber over $x \in M$ is $Cl(E_x)$, the Clifford algebra corresponding to $(E_x, \langle \cdot, \cdot \rangle|_{x})$. It is generated by elements in $E_x$ subject to the relation
\begin{equation}\label{ccr}
e_1 e_2 + e_2 e_1 = \langle e_1, e_2 \rangle.
\end{equation}

\paragraph{Contravariant spinors.}

Given an isotropic splitting $\nabla: TM \Arrow E$, consider the identification $\Phi_{\nabla}: E \Arrow TM$ \eqref{fi_nabla}. For $e \in E$, define
\begin{equation}\label{s_nabla}
s_{\nabla}(e) = \pr_{T^*M}(\Phi_{\nabla}(e)).
\end{equation}
Associated to $\nabla$, there exists a representation of the Clifford bundle $Cl(E)$ (see \cite{alekseev_xu}),
$
\Pi_{\nabla}: Cl(E) \Arrow \End{\Ext{T^*M}},
$
given on generators $e \in E$ by 
$$
\Pi_{\nabla}(e) \alpha = i_{p(e)} \alpha + s_{\nabla}(e) \wedge \alpha, \, \text{ for } \alpha \in \Ext{T^*M}.
$$
Given a section $\varphi \in \Gamma(\Ext{T^*M})$, define 
\begin{equation}\label{N_nab}
\mathcal{N}_{\nabla}(\varphi) = \{e \in E\,\, | \,\, \Pi_{\nabla}(e) \varphi=0\}.
\end{equation}
The relations \eqref{ccr} imply that $\mathcal{N}_{\nabla}(\varphi)$ is an isotropic subbundle of $E$ (whenever it has constant rank).

\begin{dfn}{\em
We say that $\varphi \in \Gamma(\Ext{T^*M})$ is a \textit{pure spinor} if $\mathcal{N}_{\nabla}(\varphi)$ is a Lagrangian subbundle of $E$.
}
\end{dfn}

Dually, given a Lagrangian subbundle $L \subset E$, define
\begin{equation}\label{U_nab}
U_{\nabla}(L) = \{ \varphi \in \Ext{T^*M}\,\, | \,\, \Pi_{\nabla}(e)\varphi=0, \, \forall \, e \in L\}.
\end{equation}
It can be proven (see \cite{chevalley}) that $U_{\nabla}(L)$ defines a line bundle on $M$ called the \textit{pure spinor line bundle} associated to $L$.

If we change the splitting by a 2-form $B \in \Omega^2(M)$, the representation changes accordingly,
\begin{equation}\label{B_rep}
\Pi_{\nabla+B} \circ e^B = e^B \circ \Pi_{\nabla},
\end{equation}
where $e^B: \Ext{T^*M} \Arrow \Ext{T^*M}$ is the exterior multiplication by $e^B= \sum_{n=0}^{\infty} \frac{1}{n!} B^n$. In particular, for any Lagrangian subbundle $L\subset E$,
\begin{equation}\label{B_spin}
U_{\nabla+B}(L) = e^B (U_{\nabla}(L)).
\end{equation}

We shall omit the reference to $\nabla$ in $U_{\nabla}$ (resp. $N_{\nabla}, \, \Pi_{\nabla}$) when considering the canonical splitting $X\in TM \Arrow (X,0) \in \T M$ for the class of Courant algebroids $(\T M, \Cour{\cdot, \cdot}_H)$, $H \in \Omega^3_{cl}(M)$.

\begin{exam}\label{omega_spin}{\em
Let $\omega \in \Omega^2(M)$ be a two form. Corresponding to the Lagrangian subbundle $L_{\omega}$ given on Example \ref{2_form},
$
\varphi= e^{-\omega}
$
is a global section for the pure spinor line bundle $U(L_{\omega})$.
}
\end{exam}

\begin{exam}\label{delta_spin}{\em
Let $\Delta \subset TM$ be a distribution. The pure spinor line bundle corresponding to $L_{\Delta}$ (see Example \ref{distributions}) is $U(L_{\Delta}) = \det (\Ann{\Delta})$, the conormal bundle of $\Delta$. In the extremal cases, $\Delta = TM$ and $\Delta = \{0\}$, one has
$$
U(L_{TM}) = \wedge^0 T^*M \,\,\text{ and } \,\, U(L_{\{0\}}) = \det (T^*M).
$$
}
\end{exam}

Regarding Examples \ref{omega_spin} and \ref{delta_spin}, note that although there exists a global section for $U(L_{\omega})$,  the existence of a global section for $U(L_{\Delta})$ is obstructed by the orientability of the space of leaves (i.e. $U(L_{\Delta})$ has a global section if and only if $\Delta$ is transversally orientable). This shows that the pure spinor line bundle carries extra information about the Dirac structure.

The integrability of a Lagrangian subbundle $L \subset E$ can be encoded by its pure spinor line bundle $U_{\nabla}(L)$, where $\nabla$ is an isotropic splitting with curvature $H \in \Omega^3(M)$. One has that
\begin{equation}\label{spinor_integrability}
L \text{ is integrable} \Longleftrightarrow \,d_H (U_{\nabla}(L)) \subset \Pi_{\nabla}(Cl(E)) U_{\nabla}(L),
\end{equation}
where $d_H = d - H\wedge \cdot \,\,$ is the \textit{$H$-twisted differential}. We refer to \cite{gualt_thesis} (see also \cite{a_b_m}) for a proof.

\paragraph{Covariant spinors.}
The Clifford bundle also has a representation on multi-vector fields. The map $\op{\Pi_{\nabla}} : Cl(E) \lrrow \End{\Ext{TM}}$ is given on generators $e \in E$ by
$$
\Pi_{\nabla}^{\vphantom{}^{op}}(e)\mathfrak{X} = p(e) \wedge \mathfrak{X} + i_{s_{\nabla}(e)} \mathfrak{X}.
$$
A section $\mathfrak{X} \in \Gamma(\Ext{TM})$ is called a \textit{covariant pure spinor} if
$
\op{N_{\nabla}}(\mathfrak{X}):= \{ e \in E \,\, | \,\, \op{\Pi_{\nabla}}(e) \mathfrak{X} = 0\} 
$
is a Lagrangian subbundle of $E$. The covariant pure spinor line bundle corresponding to a Lagrangian subbundle $L \subset E$ is defined analogously,
$
\op{U_{\nabla}}(L)= \{ \mathfrak{X} \in \Ext{TM} \,\, | \,\, \op{\Pi_{\nabla}}(e) \mathfrak{X} = 0 , \,\, \forall \, e \in L \}.
$

\begin{exam}{\em
Let $\pi \in \wedge^2 TM$ be a bivector field and consider the associated map $\pi^{\sharp}: T^*M \Arrow TM$ which takes $\xi \in T^*M$ to $\pi(\xi, \cdot) \in TM$. Its graph $\Graph{\pi^{\sharp}} = \{(\pi^{\sharp}(\xi) , \xi) \,\, | \,\, \xi \in T^*M\}$ is a Lagrangian subbundle which is integrable if and only if $\pi$ defines a Poisson bracket on $M$. As Example \ref{2_form}, the corresponding covariant pure spinor line bundle $\op{U}_{\nabla}(\Graph{\pi^{\sharp}})$ has a nowhere vanishing section $e^{-\pi}$.
}
\end{exam}

Under the assumption that $M$ is orientable, the contravariant and covariant representation of $Cl(E)$ are isomorphic. Indeed, let $\nu$ be a section of determinant bundle $\det(T^*M)$. Define
\begin{equation}\label{star_map}
\begin{array}{rccl}
\mathcal{F}_{\nu}: & \Ext{TM} & \lrrow & \Ext{T^*M} \\
       &   \mathfrak{X} & \longmapsto & i_{\mathfrak{X}} \nu.
\end{array}
\end{equation}
One has
\begin{equation}\label{star_int}
\mathcal{F}_{\nu} \circ \op{\Pi_{\nabla}} = \Pi_{\nabla} \circ \mathcal{F}_{\nu}
\end{equation}
In general, $\op{\Pi_{\nabla}}$ and $\Pi_{\nabla}$ are only locally isomorphic. 

The map  $\mathcal{F}_{\nu}$ is a Fourier-type transform which exchanges exterior multiplication with (inner) derivation (see \cite{g_s_super} for an explicit formulation of this analogy and also for the proof of \eqref{star_int} in the linear algebra setting). The inverse of $\mathcal{F}_{\nu}$ is
\begin{equation}\label{star_op}
\begin{array}{rccl}
\mathcal{F}_{\mathfrak{p}}: & \Ext{T^*M} & \lrrow & \Ext{TM} \\
       &   \alpha & \longmapsto & i_{\alpha} \mathfrak{p},
\end{array}
\end{equation}
where $\mathfrak{p} \in \det(TM)$ is such that $i_{\nu}\mathfrak{p}=1$.

\paragraph{Action by automorphisms.}
Fix an isotropic splitting $\nabla$ for $E$. The group $\Aut(E)$ acts on the Clifford representation $\Pi_{\nabla}$ as we now explain. Given an element $A=(\Psi, \psi) \in \Aut(E)$, consider the associated splitting (see Remark \ref{pertub_split}) $\nabla^A = \Psi^{-1} \circ \nabla \circ \psi_{*}$ and recall the identification $A = (\psi, B)$, given in Proposition \ref{nabla_symmetries}, where $\nabla = \nabla^A + B$. 

As $\Psi$ preserves the pairing $\langle \cdot, \cdot \rangle$, it defines a bundle map $Cl(\Psi) : Cl(E) \Arrow Cl(E)$
covering $\psi$, which is a fibrewise algebra isomorphism (on generators $e \in E$, it is just $\Psi$ itself). A straightforward calculation shows that $Cl(\Psi)$ intertwines the representations $\Pi_{\nabla}$ and $\Pi_{\nabla^A}$, i.e. the diagram
$$
\begin{CD}
Cl(E) @> \Pi_{\nabla^A} >> \End{\Ext{T^*M}} \\
@V Cl(\Psi) VV        @VV (\psi^{-1})^* \circ (\,\,\cdot\,\,) \circ \,\psi^*   V\\
Cl(E) @> \Pi_{\nabla}>> \End{\Ext{T^*M}}
\end{CD}
$$
commutes for all $a \in Cl(E)$. 

Define a bundle isomorphism $\Sigma_A: \Ext{T^*M} \Arrow \Ext{T^*M}$, covering $\psi$, by
\begin{equation}\label{action_def}
\Sigma_A = (\psi^{-1})^*\circ e^{-B}.
\end{equation}
Note that if $A$ preserves $\nabla$, then $\Sigma_A$ is just $(\psi^{-1})^*$.

Using formula \eqref{B_rep} to relate $\Pi_{\nabla^A}$ and $\Pi_{\nabla}$, one obtains
\begin{equation}\label{action_eq}
\Sigma_A \circ \Pi_{\nabla}(\cdot) \circ \Sigma_A^{-1} = \Pi_{\nabla}\circ Cl(\Psi)(\cdot).
\end{equation}
The map \eqref{action_def} induces an action of $\Aut(E)$ on $\Gamma(\Ext{T^*M})$. This action was first introduced in \cite{hu_uribe}, where it was used to define an equivariant cohomology associated to any exact Courant algebroid. We refer to \cite{hu_uribe} for more details (in particular, for the question of how $\Sigma_A$ depends on $\nabla$).

\begin{rem}\label{ps_invariance}{\em
By equation \eqref{action_eq}, if $\varphi \in \Gamma(\Ext{T^*M})$ is a pure spinor, then $\Sigma_A(\varphi)$ is also a pure spinor and
$$
\mathcal{N}_{\nabla}(\Sigma_A(\varphi)) = \Psi(\mathcal{N}_{\nabla}(\varphi)).
$$
In particular, $\Psi$ leaves  $\mathcal{N}_{\nabla}(\varphi)$ invariant if and only if $\Sigma_A$ preserves the pure spinor line generated by $\varphi$.
}
\end{rem}

\begin{exam}\label{vol_invariant}{\em
Consider the canonical Lagrangian subbundle $p^*T^*M \subset E$. By Example \ref{distributions},
$
U_{\nabla}(p^* T^*M ) = \det (T^*M).
$
For any $A \in \Aut(E)$, one can directly check that $\Sigma_A (\det(T^*M)) = \det(T^*M)$, which implies that $A$ preserves $p^*T^*M$.
}
\end{exam}
\vspace{5pt}

The associated infinitesimal action of $\Der(E)$ on $\Gamma(\Ext{T^*M})$ is given by
$$
(X, B) \cdot \alpha= \Lie_X \alpha + B \wedge \alpha,
$$
for $(X,B) \in \Der(E)$ and $\alpha \in \Gamma(T^*M)$.

In the covariant case, the element $A = (\Psi, \psi) \in \Aut(E)$ acts on $ \mathfrak{X} \in \Ext{TM}$ by
$
\op{\Sigma_A} (\mathfrak{X}) = \psi_* (i_{e^{-B}} \,\mathfrak{X}).
$
If $\psi$ preserves a volume form $\nu \in \det(T^*M)$, then the isomorphism $\mathcal{F}_{\nu}$ \eqref{star_map} intertwines $\op{\Sigma_A}$ and $\Sigma_A$.

\section{Generalized Reduction.}

\subsection{Isotropic lifted actions.}

Let $E$ be an exact Courant algebroid over $M$ and $G$ be a compact, connected Lie group acting on $M$. For $g \in G$, let $\psi_g \in \Diff(M)$ be the corresponding diffeomorphism.

\begin{dfn}\label{lifted_def}\cite{b_c_g}\em
A \textit{$G$-lifted action} on $E$ is a pair $(\mathcal{A}, \chi)$, where $\mathcal{A} :G \Arrow \Aut(E)$ is an homomorphism and $\chi: \frakg \Arrow \Gamma(E)$ is a bracket preserving map satisfying:
\begin{itemize}
\item[(1)] for $g\in G$, the corresponding automorphism $\mathcal{A}_g$ covers $\psi_g$;
\item[(2)] the infinitesimal action $\rho: \frakg \Arrow \Der(E)$ associated to $\mathcal{A}$ admits a factorization
$$
\xy
(0,10)*++^{\frakg}="g"; (20,10)*+^{\Der(E)}="der"; (20,0)*+^{\Gamma(E)}="gamma";
{\ar@{->}^{\rho}"g"; "der"}; 
{\ar@{->}_{\chi}"g"; "gamma"};
{\ar@{->}_{ad}"gamma"; "der"};
\endxy
$$
\end{itemize}

\end{dfn}
The $G$-lifted action $(\mathcal{A}, \chi)$ is said to be \textit{isotropic} if 
$$
\langle \chi(u), \chi(u) \rangle = 0, \,\, \forall \, u \in \frakg.
$$
An isotropic splitting $\nabla: TM \Arrow E$ is said to be \textit{$G$-invariant} if $\mathcal{A}_g$ preserves $\nabla$, for every $g\in G$. By compactness of $G$, there always exists $G$-invariant splittings for $E$.

\begin{exam}\label{model}{\em
Let $H \in \Omega^3(M)$ be a closed 3-form. The map
\begin{equation}\label{standard_action}
\mathcal{A} : g \in G \longmapsto 
\mathcal{A}_g = \left(
\begin{array}{cc}
(\psi_g)_* & 0 \\
0 & \psi_{g^{-1}}^*
\end{array}
\right)
\end{equation}
takes value in $\Aut(\T M, \Cour{\cdot, \cdot}_H)$ if and only if $H$ is invariant. In this case, $\mathcal{A}$ defines an homomorphism. The infinitesimal action $\rho: \frakg \Arrow \Der(\T M)$ corresponding to $\mathcal{A}$ is given, for $Y+\eta \in \Gamma(\T M)$, by
$$
\rho(u) (Y+\eta) = [u_M, Y] + \Lie_{u_M} \eta,
$$
where $u_M$ is the infinitesimal generator of the $G$-action on $M$ corresponding to $u \in \frakg$.
The existence of bracket preserving map $\chi: \frakg \Arrow \Gamma(\T M)$ such that $\rho= \ad\circ \chi$ is equivalent to the existence of a linear map $\xi: \frakg \Arrow \Omega^1(M)$ such that
\begin{itemize}
\item [(i)] $ \rho(u) = \Cour{u_M + \xi(u), \cdot}\,\, ( \Leftrightarrow d\xi(u) -i_{u_M} H = 0)$;
\item[(ii)] $\xi([u,v]) = \Lie_{u_M} \xi(v)$,
\end{itemize}
for every $u, \, v \in \frakg$. In this case, $\chi(u) = u_M + \xi(u)$. Furthermore, $\langle \chi(u), \chi(u) \rangle = 0$ if and only if
\begin{itemize}
\item[(iii)] $i_{u_M} \xi(u) =0$.
\end{itemize}
Following \cite{lin_tolman}, we call $\xi: \frakg \Arrow \Omega^1(M)$ \textit{the moment one-form} for $(\mathcal{A}, \chi)$. 
}
\end{exam}

\begin{rem}\label{gen_lifted}{\em
By choosing a splitting $\nabla$ for $E$, any isotropic $G$-lifted action $(\mathcal{A}, \chi)$ corresponds to an isotropic $G$-lifted action $(\mathcal{A}_{\nabla}, \chi_{\nabla})$ on $(\T M, \Cour{\cdot, \cdot}_{H})$, where $H \in \Omega^3(M)$ is the curvature of $\nabla$. It follows from Proposition \ref{nabla_symmetries} that if $\nabla$ is $G$-invariant, then $(\mathcal{A}_{\nabla}, \chi_{\nabla})$ is of the form considered in Example \ref{model}.
If $\xi_{\nabla}: \frakg \Arrow \Omega^1(M)$ is the moment one-form corresponding to $\mathcal{A}_{\nabla}$, then
\begin{equation}\label{moment_nabla}
\chi(u) = \nabla u_M + p^* \xi_{\nabla}(u), \,\, \forall \,u \in\frakg.
\end{equation}
Moreover,
\begin{equation}\label{moment_B}
\xi_{\nabla+B}(u) = \xi_{\nabla}(u) - i_{u_M} B, \, u \in \frakg,
\end{equation}
is the moment one-form corresponding to $\mathcal{A}_{\nabla+B}$, for $B\in \Omega^2(M)$.
}
\end{rem}

\paragraph{Equivariant cohomology.}
The obstruction to lift a $G$-action on $M$ to an isotropic lifted action on $\T M$ lives on the equivariat cohomology of $M$. Indeed, let 
$$
\Omega_G(M) =  (S(\frakg^*) \otimes \Omega(M))^{G}
$$
be the Cartan model for the equivariant cohomology of $M$ and $d_G: \Omega_{G}(M) \Arrow \Omega_{G}(M)$ be the Cartan differential (see \cite{g_s_super}).

\begin{prop}\cite{b_c_g}\label{equiv_interp}
Let $H \in \Omega^3(M)$ be a closed form and let $\mathcal{A}$ be the map \eqref{standard_action}. There exists a  map $\chi:\frakg \Arrow \Gamma(\T M)$ such that $(\mathcal{A}, \chi)$ defines an isotropic $G$-lifted action $\mathcal{A}$ on $(\T M, \Cour{\cdot, \cdot}_H)$ if and only if there exists $\xi \in S^1(\frakg^*) \otimes \Omega^1(M)$ such that  $H + \xi \in \Omega^3_G(M)$ and $d_G(H +\xi) = 0$. In this case, $\xi$ is the moment one-form of $(\mathcal{A}, \chi)$.
\end{prop}

The next lemma (proven in Lemma 8 of \cite{lin_tolman2}) provides an useful tool to simplify the description of an isotropic $G$-lifted action. In what follows, we will extend a bit our setting by considering an invariant submanifold $N$ of $M$ on which $G$ acts freely; denote by $j: N \Arrow M$ its inclusion. 

\begin{lem}\label{lin_lemma}
Suppose that $G$ acts on $N \subset M$ freely and let $\theta \in \Omega^1(N, \frakg)$ be a connection. Let $H \in \Omega^3(M)$ be a 3-form and $\xi \in S^1(\frakg^*) \otimes \Omega^1(M)$ such that $H+\xi \in \Omega^3_G(M)$. If $\,d_G(H+\xi)=0$, then 
\begin{equation}\label{B_def}
B(X, Y) = \langle \xi(\theta(Y)), \theta(X)_N - X \rangle + \langle \xi(\theta(X)), Y \rangle,
\end{equation}
for $X,\, Y \in T_xN, \, x \in N$, defines an element of $\Omega^2(N)^G$ such that 
$$
i_{u_N} B= j^*\xi(u), \,\, \forall \, u \in \frakg.
$$
In this case, $j^*H+dB=j^*(H+\xi)+ d_G B$ is a basic 3-form on $N$.
\end{lem}

\begin{dfn}{\em
Let $(\mathcal{A}, \chi)$ be an isotropic $G$-lifted action and $\nabla: TM \Arrow E$ an invariant isotropic splitting. We say that $\chi_{\nabla}: \frakg \Arrow \Gamma(\T M)$ is \textit{purely tangent on} $N$ if 
\begin{equation}\label{ad_one}
j^*\xi_{\nabla} = 0.
\end{equation}
In other words,  $\chi_{\nabla}$ is purely tangent on $N$ if it is given by vector fields up to restriction to $N$.
}
\end{dfn}
\vspace{5pt}

\begin{rem}\label{ad_one_prop}\em
Let $\nabla$ be any invariant splitting for which $\chi_{\nabla}$ is purely tangent on $N$ and let $H \in \Omega^3(M)$ be its curvature. One has that
$$
d_G(j^*H) = d_G(j^*(H + \xi_{\nabla})) = 0.
$$
Hence, $j^*H$ is basic on $N$. The existence of such splittings follows directly from Lemma \eqref{lin_lemma}. Indeed, if $\nabla_0$ is any invariant splitting, then $\nabla = \nabla_0 + \widehat{B}$, where $\widehat{B} \in \Omega^2(M)^G$ is any invariant extension of \eqref{B_def} (e.g. obtained by choosing an equivariant tubular neighborhood of $N$), turns  $\chi_{\nabla}$ into a purely tangent action on $N$. 
\end{rem}

\subsection{Reduction of Dirac structures.}

Let $E$ be an exact Courant algebroid over $M$ and $G$ a compact, connected Lie group. At the outset, let us fix an isotropic $G$-lifted action $(\mathcal{A}, \chi)$ on $E$ and suppose we are given an invariant submanifold of $M$, $j:N \Arrow M$,  on which $G$ acts freely. For $x \in N$, define
\begin{equation}\label{isotropic}
K_x= \{\chi(u)(x) \,\, | \,\,  u \in \frakg\} + p^*(\Ann{T_x N}) \subset  E_x
\end{equation}

\begin{lem}\cite{b_c_g2}
$K$ is an equivariant isotropic subbundle of $E|_N$.
\end{lem}

For a proof, we refer to Proposition 2.3 in \cite{b_c_g2} (note that in their case, $N = \mu^{-1}(0)$, where $\mu: M \Arrow \frakh^*$ is an equivariant map and $\frakh$ is a $\frakg$-module. With this in mind, the proof is exactly the same). Note that, as $\mathcal{A}_g$ preserves $\langle \cdot, \cdot \rangle$, for all $g \in G$, $K\p$ is also equivariant.

\begin{rem}\label{K_nabla}\em
If we are given an invariant splitting $\nabla$ for which $\chi_{\nabla}$ is purely tangent on $N$, then, by \eqref{moment_nabla},
$$
K= \{\nabla u_M \,\, | \,\, u \in \frakg \}|_N \oplus p^* (\Ann{TN})
$$
\end{rem}
\vspace{5pt}

Using the isotropic subbundle $K$, we can define an exact Courant algebroid over the quotient manifold $M_{red}= N/G$.

\begin{thm}[Bursztyn-Cavalcanti-Gualtieri \cite{b_c_g}]\label{reduction_courant}
Let $E_{red}$ be the bundle over $M_{red}$ defined by
$$
E_{red} = \left.\frac{K\p}{K}\right/G.
$$
The bilinear form $\langle \cdot, \cdot \rangle$, the anchor map $p: E \Arrow TM$ and the bracket $\Cour{\cdot, \cdot}$ on $\Gamma(E)$ induce a bilinear form $\langle \cdot, \cdot \rangle_{red}$, an anchor $p_{red}: E_{red} \Arrow M_{red}$ and a bracket on $\Gamma(E_{red})$ which turn $E_{red}$ into an exact Courant algebroid.
\end{thm}

It is exactly inside $E_{red}$ where the reduced Dirac structures will be found.

\begin{thm}[Bursztyn-Cavalcanti-Gualtieri \cite{b_c_g}]\label{reduction_Dirac}
Let $L \subset E$ be an invariant Dirac structure.  If $L|_{N} \cap K$ has constant rank, then
\begin{equation}\label{reduced_dirac}
L_{red} = \left. \frac{L|_N\cap K\p + K}{K} \right /G \subset E_{red}
\end{equation}
defines a Dirac structure.
\end{thm}

The main purpose of this work is to understand the relation between the pure spinor line bundles of $L$ and $L_{red}$. For this, we shall need a better description of $L_{red}$. First, let us describe the anchor $p_{red}$ and the bilinear form $\langle \cdot, \cdot \rangle_{red}$. Let $q: N \Arrow M_{red}$ be the quotient map and $x \in N$. For $k\p \in K\p_x$, let $[k\p+K]$ denote its $G$-orbit in $E_{red}|_{q(x)}$. For $k_1\p, k_2\p \in K\p_x$ and $\eta \in T_{q(x)}^*M_{red}$, one has
\begin{equation}\label{g_red}
\langle [k_1\p+K], [k_2\p +K] \rangle_{red} = \langle k_1\p, k_2\p \rangle
\end{equation}
and
\begin{equation}\label{p_red}
p_{red}^* \,\eta = [p^* dq_x^* \eta + K].
\end{equation}

\begin{rem}\label{abuse}{\em
For \eqref{p_red} to make sense, one should choose a right splitting for
$$
0 \lrrow \Ann{T_xN} \hookrightarrow \ T_x^*M \stackrel{dj_x^*} \lrrow T_x^*N \lrrow 0
$$
in order to consider $dq^*_x \eta$ as an element of $T_x^*M$. Note that \eqref{p_red} does not depend on this choice, as any two splittings differ by an element of $\Ann{T_xN}$ and $p^*\Ann{T_xN} \subset K$. In the following, 
we shall implicitly assume such a right splitting is chosen.
}
\end{rem}

Fix an invariant splitting $\nabla$ for which $\chi_{\nabla}$ is purely tangent on $N$. We claim that $\nabla(TN) \subset K\p$. Indeed, let $X \in TN$ and $\xi_{\nabla}: \frakg \Arrow \Omega^1(M)$ be the moment one-form of $(\mathcal{A}_{\nabla}, \chi_{\nabla})$. A general element $k \in K$ has the form (see \eqref{isotropic}),
$$
k= \nabla u_M + p^*(\xi_{\nabla}(u) +\eta),
$$
for $u\in \frakg$ and $\eta \in \Ann{TN}$. Hence,
$$
\langle \nabla X , k   \rangle  = \langle \nabla X , \nabla u_M + p^*(\xi_{\nabla}(u) +\eta) \rangle = i_X (\xi_{\nabla}(u) + \eta) = 0,
$$
as $j^*(\xi_{\nabla}(u) +\eta) =0$. This proves that $\nabla X \in K\p$. 

\begin{lem}
Let $\nabla$ be any invariant splitting and let $\theta \in \Omega^1(N, \frakg)$ be a connection on $N$. Define $\nabla_{red}: TM_{red} \Arrow E_{red}$ by
\begin{equation}\label{nabla_red}
\nabla_{red} \,dq(X) = [\nabla X + p^*i_XB + K], \, X \in TN,
\end{equation}
where $B \in \Omega^2(N)$ is given by \eqref{B_def} for $\xi=\xi_{\nabla}$. One has that $\nabla_{red}$ is an isotropic splitting for $E_{red}$. Moreover, its curvature $H_{red} \in \Omega^3(M_{red})$ is determined by
\begin{equation}\label{reduced_curvature}
q^*H_{red} = j^*H + dB,
\end{equation}
where $H$ is the curvature of $\nabla$.
\end{lem}

\begin{proof}
First of all, note that $\nabla_{red}$ is well-defined. Indeed, as
$$
\nabla X + p^*i_XB + K= (\nabla + \widehat{B})X + K,
$$
for any invariant extension $\widehat{B} \in \Omega^2(M)$ of $B$,  the result follows from Remark \ref{ad_one_prop}. Now, we have to prove that $p_{red} \circ \nabla_{red} = id$ and $\nabla_{red}(TM_{red})$ is a Lagrangian subbundle of $E_{red}$. For this, let $x \in N$, $X\in T_{x}N$ and $\eta \in T_{q(x)}^*M_{red}$,
$$
\begin{array}{rl}
\eta(dq_x(X)) = & \langle p^* dq_x^* \eta , \nabla X + p^*i_XB \rangle \vspace{3pt}\\
= & \langle [p^*dq_x^* \eta +K], [\nabla X + p^*i_XB + K] \rangle_{red}\vspace{3pt} \\
= & \langle p_{red}^* \eta, \nabla_{red} \,dq_x (X) \rangle \vspace{3pt}\\
= & \eta(p_{red} \circ \nabla_{red} \, dq_x(X)) .
\end{array}
$$
As $\eta$ and $X$ are arbitrary, it follows that $\nabla_{red} \circ p_{red}=id$. The fact that $\nabla_{red}(TM_{red})$ is isotropic is a direct consequence of \eqref{g_red} and \eqref{nabla_red}. For a proof of the last statement, see Proposition 3.6 and the discussion after Proposition 3.8 in \cite{b_c_g}.
\end{proof}

We are now able to give an alternative characterization of $L_{red}$ for $L \subset E$ an invariant Dirac structure. As before, fix an invariant splitting $\nabla$, a connection $\theta \in \Omega^1(N, \frakg)$ and let $B \in \Omega^2(N)$ be defined as \eqref{B_def} for $\xi=\xi_{\nabla}$. For $x \in N$, define 
\begin{equation}\label{res_dirac}
\mathfrak{B}j(L_x) = \{X + dj_x^*\beta - i_X B \in \T_x N \,\, | \,\, \nabla X + p^* \beta \in L_x\}.
\end{equation}

\begin{prop}\label{L_char}
Let $x \in N$ and  $Y+ \eta \in \T_{q(x)} M_{red}$. One has that
$$
\nabla_{red} Y + p^*_{red}\, \eta \in L_{red}|_{q(x)} \Longleftrightarrow \exists\, X\in T_xN \text{ s.t. } 
\left\{ 
\begin{array}{l}
Y = dq_x(X);\\
X+ dq_x^* \eta \in \mathfrak{B}j(L_x).
\end{array}\right.
$$
\end{prop}
\vspace{2pt}
\begin{proof}
The proof follows from a straightforward application of the relevant definitions. Let $X'\in T_xN$ such that $dq_x(X') = Y$. From the definitions of $p_{red}$ \eqref{p_red} and $\nabla_{red}$ \eqref{nabla_red}, one has
$$
\nabla_{red} Y + p^*_{red} \,\eta  = [\underbrace{\nabla X' + p^*(i_{X'}B + dq_x^* \eta)}_{k^{\perp}} + K].
$$
From the definition of $L_{red}$ \eqref{reduced_dirac}, it follows that
$$
\nabla_{red} Y + p^*_{red} \,\eta \in L_{red}|_{q(x)} \Longleftrightarrow \exists\, k \in K_x \text{ such that } k^{\perp} + k \in L_x.
$$
By Remark \ref{K_nabla}, $k= \nabla u_M|_x + p^*(i_X B + \alpha)$, for some $u\in \frakg$ and $\alpha \in \Ann{T_xN}$ (because $\chi_{\nabla + B}$ is purely tangent on $N$). Define $X= X'+u_N(x) \in T_xN$ and $\beta = i_XB + dq_x^*\eta + \alpha \in T_x^*M$. Then,
$$
\nabla X + p^* \beta=k^{\perp}+ k  \in L_x.
$$
Now, observe that $dq_x(X) = dq_x(X') = Y$. Also, as  $\alpha$  belong to $\Ann{T_xN}$, one has $dj_x^* \beta - i_XB = dq_x^*\eta$ as we wanted to prove.
\end{proof}

\begin{rem}{\em
Proposition \ref{L_char} allows us to relate the reduction procedure of \cite{b_c_g} with the one developed by P. Xu and M. Stienon in \cite{stie_xu}. Indeed, using an invariant splitting $\nabla$ for which $\chi_{\nabla}$ is purely tangent on $N$ to identify $E$ with $\T M$ and $E_{red}$ with $\T M_{red}$, Proposition \ref{L_char} shows that $L_{red}$ is the same reduced Dirac structure as defined in \cite{stie_xu}.
}
\end{rem}

We finish this subsection by providing an interpretation of Proposition \ref{L_char} using the notion of forward Dirac map.

\begin{dfn}\cite{bursztyn_radko} {\em
Let $M_1, M_2$ be manifolds and $H_i \in\Omega^3(M_i), \,i=1,2,$ be closed 3-forms. Let $L_i \in (\T M_i, \Cour{\cdot, \cdot}_{H_i}), i=1,2,$ be Dirac structures. A map $f: (M_1, H_1,L_1) \Arrow (M_2, H_2, L_2)$ is said to be \textit{forward Dirac map} if , for $x\in M_1$,
$$
L_2|_{f(x)} = \{df_x(X)+ \eta \in \T_{f(x)}M_2 \,\, | \,\, X+ df_x^*\eta \in L_1|_x\}.
$$
We further assume that $f^*H_2= H_1$.
}
\end{dfn}

Given $L \subset E$ an invariant Dirac structure, we choose an invariant splitting $\nabla$ for which $\chi_{\nabla}$ is purely tangent on $N$ and let $H\in \Omega^3(M)$ be its curvature. If $\mathfrak{B}j(L_x)$ is smooth as a vector bundle, it defines a Dirac structure with respect to the $j^*H$-twisted Courant (see  \cite{courant}). It is called the \textit{restriction of $L$ to $N$}. In this case, Proposition \ref{L_char} says that
$$
q: (N, j^*H, \mathfrak{B}j(L)) \lrrow (M_{red}, H_{red}, \Phi_{\nabla_{red}}(L_{red})) 
$$
is forward Dirac, where $H_{red} \in \Omega^3(M_{red})$ is the curvature of $\nabla_{red}$ and $\Phi_{\nabla_{red}} : E_{red} \Arrow \T M_{red}$ is the identification \eqref{fi_nabla}.

\section{Reduction via pure spinors.}

Let $E$ be an exact Courant algebroid over $M$ and $G$ a compact, connected Lie group acting on $M$. As usual, we consider an invariant submanifold $N\subset M$ on which $G$ acts freely and an isotropic $G$-lifted action $(\mathcal{A}, \chi)$ on $E$ with inclusion $j: N\Arrow M$ and quotient map $q: N \Arrow M_{red}$. Let $L \subset E$ be an invariant Dirac structure. Throughout this section, we fix this reduction setting and proceed to investigate the relation between the pure spinor line bundles of $L$ and $L_{red}$.

\subsection{Interlude.}

\subsubsection{Linear algebra}

In the following, let us fix an invariant splitting $\nabla$, a connection $\theta \in \Omega^1(N, \frakg)$ and let $B \in \Omega^2(N)$ be the 2-form \eqref{B_def} for $\xi=\xi_{\nabla}$. We begin by giving a more detailed description of $U_{\nabla}(L_x)$. Let $\mathcal{S}_x \subset T_xM$ be the image of $L_x$ under the anchor map $p$. Define $\omega_{\mathcal{S}} \in \wedge^2  \mathcal{S}^*$ by
\begin{equation}\label{omega_S}
\omega_{\mathcal{S}}(X,Y) = \xi(Y),
\end{equation}
where $\xi \in T_x^*M$ is such that $\nabla X + p^*\xi \in L_x$. The fact that $L$ is isotropic implies that $\omega$ is antisymmetric.

\begin{rem}{\em
Definition \eqref{omega_S} does not depend on the choice of $\xi$. This follows from the fact that, for every $p^*\eta \in L_x \cap p^*(T_x^*M)$, 
$$
\eta (Y) = 0 , \,\, \forall \,Y \in \mathcal{S}_x.
$$
}
\end{rem}

\begin{prop}\cite{chevalley}\label{ps_description}
Let $0\neq \Omega \in \det(\Ann{\mathcal{S}_x}) \subset \Ext{T_x^*M}$. One has that
\begin{equation}\label{ps_eq}
\varphi_x= e^{-\omega}\wedge \Omega
\end{equation}
is a non-zero generator of $U_{\nabla}(L_x)$, where $\omega \in \wedge^2 T_x^*M$ is any extension of $\omega_{\mathcal{S}}$.
\end{prop}

Note that as $L$ is isotropic, we have
\begin{equation}\label{Ann_S}
\Ann{\mathcal{S}_x} = L_x \cap p^*(T_x^*M).
\end{equation}

In the covariant case, there is a analogous description of $\op{U}_{\nabla}(L_x)$ obtained by exchanging $T_x^*M$ with $T_xM$. In this case, the covariant pure spinor line of $L_x$ is generated by
$ e^{-\pi} \wedge \mathfrak{X}$, where $\pi \in \wedge^2 T_xM$ is a bivector and $0 \neq \mathfrak{X} \in \det(L_x \cap \nabla T_x M)$. 

We are now able to relate the pure spinor line bundles of $L$ and $L_{red}$ at the linear algebra level. Recall the isotropic subbundle $K \subset E|_N$ \eqref{isotropic} associated to the isotropic $G$-lifted action $(\mathcal{A}, \chi)$. 

\begin{thm}\label{main_linear}
 For $\varphi_x \in U_{\nabla}(L_x)$,
\begin{equation}\label{main_transv}
\op{\varpi}_x := dq_x \circ \mathcal{F}_{\mathfrak{p}} (e^B \wedge dj_x^* \varphi_x) \neq 0 \Leftrightarrow L_x\cap K_x = 0,
\end{equation}
where $\mathcal{F}_{\mathfrak{p}}: \Ext{T_x^*N} \Arrow \Ext{T_xN}$ is the star map \eqref{star_op} corresponding to $\mathfrak{p} \in \det(T_xN)$. In this case, $\op{\varpi}_x$ is a generator of the covariant pure spinor line $\op{U}_{\nabla_{red}}(L_{red}|_{q(x)})$.
\end{thm}

We need a lemma first. Recall the definition \eqref{res_dirac} of $\mathfrak{B}j(L_x) \subset \T_x N$.

\begin{lem}\label{bcg_lemma}
For $\varphi_x \in U_{\nabla}(L_x)$, one has
$$
dj_x^* \varphi_x \neq 0  \Leftrightarrow L_x \cap p^*(\Ann{T_xN}) =0.
$$
In this case, $e^B \wedge dj_x^* \varphi_x$ is a generator of the pure spinor line $U(\mathfrak{B}j(L_x))$.
\end{lem}

We refer to Proposition 1.5 in \cite{a_b_m} for a proof.

\begin{proof}[Proof of Theorem \ref{main_linear}]
Let us first suppose that $dj_x^* \varphi_x \neq 0$. By \eqref{star_int} and Lemma \eqref{bcg_lemma}, $\mathcal{F}_{\mathfrak{p}}(e^B \wedge dj_x^* \varphi_x)$ is a covariant pure spinor for $\mathfrak{B}j(L_x)$. So, $\mathcal{F}_{\mathfrak{p}}(e^B \wedge dj_x^* \varphi_x) = e^{-\pi} \wedge \mathfrak{X}$, where $\mathfrak{X} \in \det(\mathfrak{B}j(L_x)\cap T_xN)$ and $\pi \in \wedge^2 T_xM$.
Hence, 
$$
\op{\varpi}_x=dq_x \circ \mathcal{F}_{\mathfrak{p}}( e^B \wedge dj^*_x \varphi_x) = 0 \Leftrightarrow dq_x(\mathfrak{X}) = 0 \Leftrightarrow \mathfrak{B}j(L_x) \cap \ker(dq_x)\neq 0.
$$
By the definition of $\mathfrak{B}j(L_x)$ \eqref{res_dirac}, one has that $u_N(x) \in \mathfrak{B}j(L_x) \cap \ker(dq_x)$, for $u \in \frakg$, if and only if there exists $\beta \in \Ann{T_xN}$ such that $\nabla u_N(x) +p^* \beta \in L_x$ and $dj_x^*\beta = i_{u_N(x)} B$. By construction, $i_{u_N(x)}B = dj_x^*\xi_{\nabla}(u)$, which implies that $\beta - \xi_{\nabla}(u)|_x \in \Ann{T_xN}$. Therefore,
$$
\nabla u_N(x) + p^* \beta = \chi_{\nabla}(u)|_x + p^*(\beta - \xi_{\nabla}(u)|_x) \in L_x \cap K_x.
$$
This proves \eqref{main_transv}. Now, let us show that the assumption that $dj_x^* \varphi_x \neq 0$ can be dropped. If $\op{\varpi_x} \neq 0$, then $dj_x^* \varphi_x \neq 0$ and the argument above implies that $L_x \cap K_x =0$. On the other hand, if $L_x \cap K_x=0$, then, in particular, $L_x \cap p^* \Ann{T_xN} = 0$ which implies, by Lemma \ref{bcg_lemma}, that $dj_x^*\varphi_x \neq 0$. Again, the argument above proves that $\op{\varpi_x} \neq 0$. This concludes the proof of \eqref{main_transv}. 

To prove that $0\neq \varpi_x$ generates $U_{\nabla_{red}}(L_{red})$, let $(Y,\eta) \in \T_{q(x)} M_{red}$ such that $e:=\nabla_{red} Y + p_{red}^* \eta \in L_{red}$. By Proposition \ref{L_char},  there exists $X\in T_xN$ such that $Y=dq_x(X)$ and $(X, dq^*_x \eta) \in \mathfrak{B}j(L_x)$. Therefore,
$$
\begin{array}{rl}
\op{\Pi_{\nabla_{red}}}(e) \op{\varpi_x} = & Y \wedge \op{\varpi_x} + i_{\eta} \op{\varpi_x} \\
        = & dq_x(X \wedge \mathcal{F}_{\mathfrak{p}} (e^B \wedge dj^*_x \varphi_x) + i_{dq_x^* \eta}\, \mathcal{F}_{\mathfrak{p}} (e^{B}\wedge dj^*_x\varphi_x)) \\
= & 0,
\end{array}
$$
as $\mathcal{F}_{\mathfrak{p}} (dj_x^* \varphi_x)$ is a covariant pure spinor for $\mathfrak{B}j(L_x)$. This proves that $\op{\varpi}_x$ belongs to $\op{U}_{\nabla_{red}}(L_{red}|_x)$ as we wanted.
\end{proof}

In order to obtain a contravariant pure spinor corresponding to $L_{red}$, we have to choose, besides $\mathfrak{p} \in \det{T_x N}$, an element $\nu \in \det(T_{q(x)}^*M_{red})$.
In this case,
$$
\varpi_x=\mathcal{F}_{\nu}(\op{\varpi_x}) = \mathcal{F}_{\nu} \circ dq_x \circ \mathcal{F}_{\mathfrak{p}} (e^B \wedge dj^*_x \varphi_x) \in \Ext{T_{q(x)}M_{red}}
$$
is a contravariant pure spinor for $L_{red}|_{q(x)}$, where $\mathcal{F}_{\nu}: \Ext{T_{q(x)}M_{red}} \Arrow \Ext{T_{q(x)}^*M_{red}}$ is the Fourier-type map \eqref{star_map}. Consider
$$
\begin{array}{rccl}
    & \det(T_xN) \otimes \det(T_{q(x)}^*M_{red}) &   \lrrow & \End{\Ext{T_x^*N}, \,\Ext{T_{q(x)}^*M_{red}}}\\
    & (\mathfrak{p}, \nu) & \longmapsto & \mathcal{F}_{\nu} \circ dq_x \circ \mathcal{F}_{\mathfrak{p}}.
\end{array}
$$
Under the isomorphism between $\det(T_xN) \otimes \det(T_{q(x)}^*M_{red})$ and $\det(\ker{dq_x})$ given by
$$
\mathfrak{p}\otimes \nu \mapsto \delta_x: = (-1)^{n(n-r)}\mathcal{F}_{\mathfrak{p}} (dq_x^* \nu),
$$
where $n= \dim(N)$ and $r= \dim(G)$,  $\mathcal{F}_{\nu} \circ dq_x \circ \mathcal{F}_{\mathfrak{p}}$ corresponds to a map $C_{\delta_x}$ given only in terms of $\delta_x$. To see the expression of $C_{\delta_x}$, fix a basis $\{\xi^1, \dots, \xi^n\}$ of $T_x^*N$ such that $\{\xi^1, \dots, \xi^{n-r}\}$ generates $\Ann{\ker{dq_x}}$. Any element of $ \Ext{T_x^*N}$ is a sum of forms of the type
$$
dq_x^* \alpha \wedge \xi^I, \,\, I=\{i_1 < \cdots < i_k\} \subset \{n-r+1, \cdots, n\}, \, \alpha \in T_{q(x)}^*M_{red}.
$$
The map $\mathcal{F}_{\nu} \circ dq_x \circ \mathcal{F}_{\mathfrak{p}}$ is equal to
\begin{equation}\label{first_step_push}
\begin{array}{rccl}
C_{\delta_x}: &  \Ext{T_x^*N} & \lrrow & \Ext{T_{q(x)}M_{red}} \vspace{2pt}\\
              & dq_x^* \alpha \wedge \xi^I &  \longmapsto & 
\left\{
\begin{array}{ll}
0, & \text{ if } I \neq \{n-r+1, \cdots, n\}\\
(i_{\xi^I} \delta_x) \, \alpha, & \text{ if } I = \{n-r+1,\cdots,n\}.
\end{array}
\right. \\
\end{array}
\end{equation}

For future reference, we state the contravariant version of Theorem \ref{main_linear}.

\begin{thm}\label{main_linear_contra}
For  $\varphi_x \in U_{\nabla}(L_x)$, 
$$
\varpi_x = C_{\delta_x} (e^B \wedge dj_x^*\varphi_x) \neq 0 \Leftrightarrow L_x\cap K_x = 0,
$$
where $C_{\delta_x}: \Ext{T_x^*N} \Arrow \Ext{T^*_{q(x)}M_{red}}$ is the map \eqref{first_step_push} corresponding to $\delta_x \in \det(\ker(dq_x))$. In this case, $\varpi_x$ is a generator of the contravariant pure spinor line $U_{\nabla_{red}}(L_{red}|_{q(x)})$.
\end{thm}

\subsubsection{Push-forward.}

As $G$ acts freely on $N$, the quotient map $q: N \Arrow M_{red}$ is a $G$-principal bundle. By assumption, $G$ is a compact, connected Lie group and, therefore, we can consider the push-forward map
$$
q_*: \Omega(N) \lrrow \Omega(M_{red}).
$$
We recall its definition. Let $\mathcal{U} \subset M_{red}$ be an open set such that $N|_{q^{-1}(\mathcal{U})}$ is trivial and let $\pr_G: N|_{q^{-1}(\mathcal{U})} \Arrow G$ be the projection on the fiber. Locally, any differential form on $N$ is a sum of two types of differential forms:
$$
f \, q^* \beta \wedge \pr_G^*\nu, \text{ where } \beta \in \Omega(M_{red})
\text{ and } \left\{
\begin{array}{ll}
\nu\in \Omega^r(G), & \text{ type (I)};\\
\nu\in \Omega^k(G), \,k< r, & \text{ type (II)}
\end{array}\right.
$$
with $f \in C^{\infty}(q^{-1}(\mathcal{U}))$ and $r= \dim(G)$. The \textit{push-forward}  $q_*: \Omega(N) \Arrow \Omega(M_{red})$ is locally defined by
\begin{equation}\label{push_def}
\alpha=f \, q^* \beta \wedge \pr_G^*\nu \longmapsto 
\left\{
\begin{array}{ll}
\left(\int_G f(\cdot, g) \, \nu\right) \, \beta,  & \text{ if } \alpha \text{ is type (I)};\vspace{3pt}\\
0,  &  \text{ if } \alpha \text{ is type (II).}
\end{array}\right. 
\end{equation}

Alternatively, one can see the push-forward map as a composition
\begin{equation}\label{push_comp}
\Omega(N) \stackrel{C_{\delta}} \lrrow \Gamma(q^* \Ext{T^*M_{red}}) \stackrel{I_{\delta}} \lrrow \Omega(M_{red}).
\end{equation}
The first map is induced by the bundle map $C_{\delta}: \Ext{T^*N} \Arrow q^*\Ext{T^*M_{red}}$ defined pointwise by \eqref{first_step_push}, where $\delta \in \Gamma(\det(\ker{dq}))$ is the image of some fixed element $\delta_{\frakg} \in \wedge^{r} \frakg$ under the natural extension 
\begin{equation}\label{inf_extension}
\upsilon: \wedge^r \frakg \lrrow \Gamma(\wedge^r TN)
\end{equation}
of the infinitesimal action $\upsilon: \frakg \Arrow \Gamma(TN)$. The second map is defined locally as
$$
f\, \alpha \mapsto \left( \int_G f(\cdot, g) \,\nu^L \right)  \alpha, \,\,\, \alpha \in \Gamma(\Ext{T^*M_{red}}|_{\mathcal{U}}), \,\,\, f \in C^{\infty}(q^{-1}(\mathcal{U})),
$$
where $\nu^L \in \det(T^*G)$ is the left-invariant volume form such that 
$
i_{\delta_{\frakg}} \nu^L (e) = 1.
$

\subsection{Main Theorem}
\subsubsection{Transversal case}

Recall the reduction setting fixed in the beginning of \S 3 and consider an invariant splitting $\nabla: TM \Arrow E$. Define a $G$-action on $\Ext{T^*M}$ by composing $\A: G \Arrow \Aut(E)$ with the ($\nabla$-dependent) action $\Sigma: \Aut(E) \Arrow \End{\Ext{T^*M}}$ defined by \eqref{action_def}. As $\nabla$ is invariant, this $G$-action on $\Ext{T^*M}$ is just the pull-back by $\varphi^*_{g^{-1}}, \, g\in G$. The fact that $L$ is invariant implies that $U_{\nabla}(L)$ is $G$-invariant (see Remark \ref{ps_invariance}).

\begin{thm}\label{main1}
Let $\theta \in \Omega^1(N, \frakg)$ be a connection and consider $B \in \Omega^2(N)$ given by \eqref{B_def} for $\xi=\xi_{\nabla}$.  Let $\varphi$ be a nowhere-zero invariant section of $U_{\nabla}(L)|_N$ over an invariant open set $\mathcal{V} \subset N$. One has that 
\begin{equation}\label{main_spin}
\varphi_{red}=q_*(e^B \wedge j^*\varphi)
\end{equation}
is a section of $U_{\nabla_{red}}(L_{red})$ over $\mathcal{V}/G$. For $x \in \mathcal{V}$,
$
\varphi_{red}|_{q(x)} = 0 \Leftrightarrow L_x \cap K_x \neq 0.
$
In particular, if $L|_N \cap K = 0$, then $\varphi_{red}$ is nowhere-zero. 
\end{thm}

\begin{proof}
First note that, by \eqref{B_spin}, $e^B \wedge j^*\varphi$ is an invariant nowhere-zero section of $\mathfrak{B}j(L)$ over $\mathcal{V}$. Now, choose $\delta_{\frakg} \in \wedge^r\frakg$ and consider its image $\delta \in \Gamma(\det(\ker(dq)))$ under the extension \eqref{inf_extension} of the infinitesimal action. Let $C_{\delta}: \Ext{T^*N} \Arrow q^*\Ext{T^*M_{red}}$ be the associated map given pointwise by \eqref{first_step_push} and consider the section $\theta$ of the pull-back bundle $q^*\Ext{T^*M_{red}}$ over $\mathcal{V}$ defined by
$$
\varpi_x = C_{\delta}\circ dj^*_x( e^B \wedge \varphi_x) \in \Ext{T_{q(x)}M_{red}}, \, x \in \mathcal{V}.
$$
By Theorem \ref{main_linear_contra}, \,$\varpi_x \in U_{\nabla_{red}}(L_{red}|_{q(x)})$ and $\varpi_x = 0 \Leftrightarrow L_x\cap K_x \neq 0$. The result will follow if we prove that $\varphi_{red}|_{q(x)}= \lambda \, \varpi_x$, for some $\lambda \in \R\backslash\{0\}$. For this, note that if  $\varpi_y = \varpi_x$, for $y$ and $x$ in the same $G$-orbit, then, by \eqref{push_comp},
$$
\varphi_{red}|_{q(x)} = q_*( e^B \wedge j^* \varphi)|_{q(x)} = I_{\delta}(\varpi)|_{q(x)}=  \left(\int_G \nu^L \right) \varpi_x,
$$
where $\nu^L$ is the left-invariant volume form on $G$ such that $i_{\delta_{\frakg}} \nu^L(e) =1$.

To prove that $x\mapsto \varpi_x$ is constant on $G$-orbits, let us first relate $\delta_y$ and $\delta_x$, for $y=\psi_g(x)$, $g\in G$. By the well-known formula,
$$
u_M(y) = d\psi_g  (Ad_{g^{-1}}(u))_M(x),
$$
it follows that $\delta_y =  d\psi_g \,(\overline{\delta}_x)$, where $\overline{\delta} \in \Gamma(\ker(dq))$ is the image of 
$$
\overline{\delta}_{\frakg}:=Ad_{g^{-1}}(\delta_{\frakg}) = \det(Ad_{g^{-1}}) \delta_{\frakg} \in \wedge^r \frakg
$$
under the natural extension \eqref{inf_extension} of the infinitesimal action. As $G$ is compact and connected, $\det(Ad_{g^{-1}}) = 1$. Therefore, 
\begin{equation}\label{delta_eq}
\delta_y =  d\psi_g \,(\delta_x).
\end{equation}
The result now follows directly from the definition \eqref{first_step_push} of $C_{\delta_{\cdot}}$ and from the relation $\varphi_y = d\psi_{g^{-1}}^* \varphi_x$.

\end{proof}

\begin{rem}\em
If $\theta_1, \theta_2 \in \Omega^1(N, \frakg)$ are two connections, then $B_1 - B_2$ is a basic 2-form. Let $\tilde{B} \in \Omega^2(M_{red})$ be such that $q^*\tilde{B}= B_1-B_2$. So,
$$
\varphi_1 = q_*\circ (e^B_1 \wedge j^*\varphi) = q_*(e^{(B_1-B_2)} \wedge e^{B_2}\wedge  j^* \varphi) = e^{\tilde{B}} \wedge \varphi_2.
$$
\end{rem}

\subsubsection{Non-transversal case.}

To have a completely general description of the pure spinor line bundle of $L_{red}$, we have to tackle the case where $L|_N \cap K$ has non-zero constant rank. We cannot apply Theorem \ref{main1} directly in this case, as the reduced pure spinor will be identically zero. We give a simple example illustrating this.

\begin{exam}\label{dist}{\em
Let $\mathcal{F}$ be a foliation on $M$ and consider a submanifold $N \subset M$. By considering a trivial Lie group $G=\{e\}$ acting  on $M$ and the Dirac structure $L_{T\mathcal{F}}$ (see Example \ref{distributions}), these data fit into our reduction setting. The isotropic subbundle $K \subset \T M|_N$ \eqref{isotropic} is just $\Ann{TN}$ and $L_{T\mathcal{F}}\cap K$ has non-zero constant rank if and only if the foliation $\mathcal{F}$ intersects cleanly, \textit{but not transversally}, with $N$. For any $\varphi \in U(L_{T\mathcal{F}}) = \det(\Ann{T\mathcal{F}})$, formula \eqref{main_spin} gives
$$
\varphi_{red}=j^*\varphi
$$
which is identically zero as $\Ann{T\mathcal{F}} \cap \Ann{TN} \neq 0$. Observe that, in any case, the reduced Dirac structure is well-defined: $(L_{T\mathcal{F}})_{red} = L_{T\mathcal{F}\cap TN}$.
}
\end{exam}

Our approach to circumvent this problem of non-transversality is to pertub our original Dirac structure to obtain the transversality condition $L|_N \cap K =0$ in such a way that the reduced Dirac structure remains the same. 

\begin{dfn}{\em
Suppose $L|_N \cap K$ has constant rank. A \textit{pertubation input} for $(L, K)$ is an invariant isotropic subbundle $D \subset E|_N$ for which 
\begin{equation}\label{pert_eq}
(L|_N\cap K)\p \oplus D = E|_N.
\end{equation}
}
\end{dfn}

Note that if $(L,K)$ satisfies the transversality condition $L|_N \cap K = 0$, the only possible pertubation input is the zero subbundle. Let us prove that pertubation inputs always exist.

\begin{lem}
Assume $L|_N \cap K$ has constant rank and let $F \subset E|_N$ be any invariant complement to $(L|_N \cap K)^{\perp}$. One has that 
\begin{equation}\label{D_eq}
D = \{e- \frac{1}{2} Ae \,\, | \,\, e \in F\}
\end{equation}
is a pertubation input, where $A: F \Arrow L|_N \cap K$ is defined by the composition 
$$
F \hspace{3pt} \vphantom{}^{\underrightarrow{e \mapsto \langle e, \cdot \rangle|_F}} \hspace{3pt}  F^*  \hspace{3pt} \vphantom{}^{\underrightarrow{\langle \cdot, \cdot \rangle^{\sharp}\vphantom{}^{-1}}} \hspace{3pt} L|_N \cap K.
$$
\end{lem}

\begin{proof}
Let us first prove that $A$ is well defined. For simplicity, call $K_L = L|_N \cap K$. By assumption, it is an invariant isotropic subbundle of $E|_N$. As $G$ is compact and preserves $\langle \cdot, \cdot \rangle$, there always exists an invariant subbundle $F \subset E|_N$ such that $K_L\p \oplus F = E|_N$. As
$$
0=(K_L\p\oplus F)\p = K_L\cap F\p,
$$
the bundle map $\langle \cdot, \cdot \rangle^{\sharp}: K_L \Arrow F^*$ given by $\langle \cdot, \cdot \rangle^{\sharp}(k) = \langle k, \cdot \rangle|_{F}$ is an isomorphism, so  $A$ is well-defined. It clearly satisfies $\langle Ae,k \rangle =\langle e,k \rangle$, for $(e,k) \in F \times_N K_L$. Hence, for $(e_1, e_2) \in F\times_N F$, one has
$$
\langle e_1-\frac{1}{2} Ae_1, e_2 - \frac{1}{2} Ae_2 \rangle = \langle e_1,e_2\rangle -\frac{1}{2}(\langle e_1,Ae_2 \rangle + \langle Ae_1, e_2 \rangle ) = 0.
$$
This proves that $D$ is an isotropic complement to $K_L$. It remains to prove that $D$ is invariant. For $g \in G$, let $\mathcal{A}_g= (\Psi_g, \psi_g) \in \Aut(E)$ and $x \in N$. For $e_1 \in F_x, \, e_2\in F_{\psi_g(x)}$, it follows from the invariance of $F$ and $\langle \cdot, \cdot \rangle$ that
$$
\begin{array}{rl}
\langle A(\Psi_g(e_1)), e_2 \rangle = & \langle \Psi_g(e_1), e_2 \rangle\vspace{3pt} \\
   = & \langle e_1, \Psi_{g^{-1}}(e_2) \rangle \vspace{3pt}\\
   = & \langle A(e_1), \Psi_{g^{-1}}(e_2) \rangle\vspace{3pt}\\
   = & \langle \Psi_g(A(e_1)), e_2 \rangle.
\end{array}
$$
By the non-degeneracy of the form, it follows that $A\circ \Psi_g = \Psi_g \circ A$. This proves that $D$ is invariant as we wanted.
\end{proof}

As the name suggests, we shall use a pertubation input $D$ for $(L, K)$ to pertub $L|_N$ in order to obtain a Lagrangian subbundle $L_D \subset E|_N$ satisfying
\begin{itemize}
\item[(i)] $L_D$ is invariant and $L_D \cap K = 0$;
\item[(ii)] $L_D \cap K\p + K = L\cap K\p + K$;
\item[(iii)] the passage from $L$ to $L_D$ is computable in the pure spinor level.
\end{itemize}

The pertubation of $L|_N$ is defined by
\begin{equation}\label{pertubed_dirac}
L_D: = L|_N \cap D\p + D.
\end{equation}
Note that if $L|_N \cap K=0$, then $L_D = L|_N$.

\begin{prop}\label{pertubation}
Suppose $L|_N \cap K$ has constant rank and let $D \subset E|_N$ be a pertubation input for $(L,K)$. The subbundle $L_D \subset E|_N$ defined by \eqref{pertubed_dirac} is Lagrangian and satisfies conditions (i) and (ii) above. Moreover, let $\varphi$ be a nowhere-zero section of $U_{\nabla}(L|_N)$ over some open set $\mathcal{V} \subset N$, where $\nabla$ is any isotropic splitting for $E$. Suppose further that there exists a section $\mathfrak{d}$ of $\det(D)$ over $\mathcal{V}$. Then
\begin{equation}\label{phi_D}
\varphi_D=\Pi_{\nabla}(\mathfrak{d}) \varphi 
\end{equation}
is a nowhere-zero section of $U_{\nabla}(L_D)$ over $\mathcal{V}$.
\end{prop}
 
\begin{rem}{\em
For any $x \in N$, the relations \eqref{ccr} and the fact that $D_x$ is isotropic imply that the subalgebra of $Cl(E_x)$ generated by $D_x$ is isomorphic to $\Ext{D_x}$. In this way, we can consider $\det(D_x) \subset Cl(E_x)$. 
}
\end{rem}

\begin{proof}
It is clear that $L_D$ is invariant as both $L|_N\cap D\p$ and $D$ are invariant subbundles of $E|_N$. The fact that $L_D$ is Lagrangian follows from
$$
L_D = [(L|_N + D)\cap D\p]\p.
$$
Indeed, one has
$$
L_D\p = (L|_N +D)\cap D\p = L|_{N} \cap D\p + D = L_D.
$$
To finish the proof of (i), note that $(L|_N\cap K)\p \cap D=0$ implies $(L|_N\cap K) \oplus D^{\perp}= E|_N$; this in turn implies that
\begin{equation}\label{L_decomposition}
L|_N=(L|_N\cap K) \oplus (L|_N\cap D^{\perp}).
\end{equation}
So,
$$
\begin{array}{rl}
E|_N=  (L|_N\cap K)\p + D =  K^{\perp} +  L|_N + D \hspace{-2pt} &  \hspace{-5pt}\stackrel{\eqref{L_decomposition}}= K\p +( L|_N\cap K+ L|_N\cap D^{\perp}) + D \\
  & \hspace{2pt}= \hspace{4pt}K\p + L|_N\cap K + L_D\\
  & \hspace{2pt}= \hspace{4pt} K\p + L_D.
\end{array}
$$
Hence,
$$
L_D\cap K = (L_D+ K\p)\p = E|_N\p = 0,
$$
by the non-degeneracy of $\langle \cdot, \cdot \rangle$. As for (ii), by (\ref{L_decomposition}),
$$
L|_N+D=L|_N \cap K + L|_N\cap D\p + D = L|_N\cap K + L_D,
$$
which implies that $L_D + K = L|_N + D + K$.
Thus
$$
L_D\cap K\p = (L_D+K)\p = (L|_N+D+K)\p \subset L|_N\cap K\p.
$$
So,
$$
L_D\cap K\p + K \subset L|_N\cap K\p + K,
$$
and, as both subspaces are Lagrangians, they have the same rank equal to $\dim(M)$. This proves (ii). It remains to prove that $\varphi_D$ is a nowhere-zero section of $U_{\nabla}(L_D)$. Let $x \in \mathcal{V}$ and $\{d_1, \dots, d_l\}$ be a basis of $D_x$ such that $\mathfrak{d}_x = d_l \wedge \cdots \wedge d_1$. For $j=1, \dots, l$, define
$$
D^j = \text{span}\{d_1, \dots, d_j\} ,\,\, \,\,L^j = L_x \cap D^j\vphantom{}\p + D^j
$$
and 
$$
\varphi^j = \Pi_{\nabla}(d_j \wedge \cdots \wedge d_1) \,\varphi_x.
$$
We shall proceed by induction. For $j=1$, 
$$
\varphi^1 = \Pi_{\nabla}(d_1) \varphi_x = 0 \Leftrightarrow d_1 \in L_x,
$$
As $L_x \cap D_x= 0$ and $d_1 \neq 0$, it follows that $\varphi^1\neq 0$. For any $e \in L_x \cap D^1\vphantom{}\p$, using \eqref{ccr}, one has
$$
\Pi_{\nabla}(e)\varphi^1 = \langle e, d_1 \rangle\, \varphi_x - \Pi_{\nabla}(d_1) \Pi_{\nabla}(e) \varphi_x = 0.
$$
Also,
$$
\Pi_{\nabla}(d_1) \varphi^1 = \Pi_{\nabla}(d_1^{\,2}) \varphi_x = 0,
$$
as $d_1^{\,2} = 2\langle d_1, d_1 \rangle =0$. This proves that $\varphi^1 \in U_{\nabla}(L^1)$ as we wanted. Assume now that $0 \neq \varphi^{j-1} \in U_{\nabla}(L^{j-1})$, for some $j \in \{2, \dots, l\}$. Then
$$
\begin{array}{rl}
\varphi^j = \Pi_{\nabla}(d_j) \varphi^{j-1} = 0 & \hspace{-8pt} \Leftrightarrow d_j \in L^{j-1} \\ 
& \hspace{-8pt} \Leftrightarrow \exists \, a_1, \dots, a_{j-1} \in \R  \text{ s.t. } d_j - \sum_{i=1}^{j-1} a_i d_i \in L_x \cap D^{j-1}\vphantom{}\p.
\end{array}
$$
As $L_x \cap D_x = 0$, it follows that if $\varphi^{j}= 0$, then $d_j = \sum_{i=1}^{j-1} a_i d_i$, which is absurd as $\{d_1, \dots, d_n\}$ is linearly independent. Hence, $\varphi^j \neq 0$. For $e \in L_x \cap D^j\vphantom{}\p$,  using \eqref{ccr}, one has
$$
\begin{array}{rl}
\Pi_{\nabla}(e) \varphi^j = & \sum_{i=1}^j (-1)^{j-i} \langle e, d_i \rangle \Pi_{\nabla}(d_j \cdots d_{i-1} d_{i+1} \cdots d_1) \varphi_x  \vspace{2pt}\\
& \hspace{-15pt}+ (-1)^{j} \Pi_{\nabla}(d_j \cdots d_1) \Pi_{\nabla}(e) \varphi_x\\
 = & 0.
\end{array}
$$
For any $i=1, \dots, j$, using \eqref{ccr} again and the fact that $D$ is isotropic,
$$
d_i d_j \cdots d_1 = (-1)^{j-i+1} d_j \cdots d_{i-1} d_i^{\,2} d_{i+1} \cdots d_1 = 0. 
$$
Hence, 
$$
\Pi_{\nabla}(d_i) \varphi^j = 0, \,\, \forall \, i \in \{1, \cdots, j\}.
$$
So, $\varphi^j \in U_{\nabla}(L^j)$ as we wanted to prove. The result follows from the observation that $\varphi^l = \varphi_D|_x$ and $L^l = L_D|_x$.
\end{proof}

\begin{rem}\label{inv_D}{\em
The map
$$
\begin{array}{cll}
U_{\nabla}(L|_N) \otimes \det(D)& \lrrow & U_{\nabla}(L_D)\\
 \varphi\otimes \mathfrak{d} & \longmapsto & \Pi_{\nabla}(\mathfrak{d}) \varphi
\end{array}
$$
is a $G$-equivariant isomorphim, where $G$ acts on $\det(D)$ via $Cl(\Psi_g)$, for $g\in G$ and it acts on pure spinor lines as explained in the beginning of \S 3.2.1. This follows from \eqref{action_eq}
}
\end{rem}

The next example shows how this pertubation procedure works in the case of Example \ref{dist}.
\begin{exam}\em
A possible pertubation input for $(L_{T\mathcal{F}},\, \Ann{TN})$ is any distribution $D \subset TM|_N$ such that $(T\mathcal{F}|_N+TN) \oplus D = TM|_N$. The pertubed bundle $L_D$ is 
$$
L_D = (T\mathcal{F}|_N + D) \oplus \Ann{T\mathcal{F}|_N + D}.
$$
The distribution $T\mathcal{F}|_N + D$ is now transversal to $TN$ and $(T\mathcal{F}|_N + D)\cap TN = T\mathcal{F}|_N \cap TN$.
The pertubed pure spinor is
$$
\varphi_D= i_{\mathfrak{d}} \varphi,
$$
where $\mathfrak{d} \in \det(D)$. Note that the contraction with $\mathfrak{d}$ kills all covector which lies in $\Ann{TN}$ and, hence, $j^*\varphi_D \neq 0$. 
\end{exam}

We are now able to state our main theorem in its most general form. As usual, fix an invariant splitting $\nabla$ for $E$ and a connection $\theta \in \Omega^1(N, \frakg)$.

\begin{thm}\label{main_geral}
Suppose $L|_N \cap K$ has constant rank and let $D$ be the pertubation input for $(L,K)$ given by \eqref{D_eq}. Let $\varphi$ be a nowhere-zero invariant section of $U_{\nabla}(L)|_N$ over an invariant open set $\mathcal{V} \subset N$ and suppose there exists an invariant section $\mathfrak{d}$ of $\det(D)$ over $\mathcal{V}$. One has that
\begin{equation}\label{reduced_spinor}
\varphi_{red}:=q_*(e^B \wedge  j^* \Pi_{\nabla}(\mathfrak{d})\varphi)
\end{equation}
is a nowhere-zero section of $U_{\nabla_{red}}(L_{red})$ over $\mathcal{V}/G$, where $B \in \Omega^2(N)$ is the 2-form given by \eqref{B_def} for $\xi = \xi_{\nabla}$.
\end{thm}

\begin{proof}
Define $\varphi_D:= \Pi_{\nabla}(\mathfrak{d})\varphi$. As $\varphi$ is a nowhere-zero section of $U_{\nabla}(L)$ over $\mathcal{V}$,  Proposition \ref{pertubation} together with Remark \ref{inv_D} says that $\varphi_D$ is a nowhere-zero invariant section of $U_{\nabla}(L_D)$ over $\mathcal{V}$, where $L_D= L|_N \cap D\p + D$. As $L_D \cap K = 0$ and $\nabla$ is $(\mathcal{A},\chi, N )$-admissible, Theorem \ref{main1} guarantees that
$$
q_*(e^B \wedge j^* \varphi_D)
$$
is a nowhere-zero section of $U_{\nabla_{red}}(L_{D, \, red})$, where 
$$
L_{D, \, red} = \left.\frac{L_D \cap K\p + K}{K} \right/G.
$$
Now, as $L_D \cap K\p + K = L|_N\cap K\p + K$, we have that $L_{D, \, red} = L_{red}$. The result now follows from (see \eqref{B_rep})
\end{proof}

\begin{rem}\em Under the assumptions of Theorem \ref{main_geral}, if $D' \subset E|_N$ is another pertubation input for $(L,K)$ and $\mathfrak{d}' \in \det(D')$ is an invariant section,  then
$$
q_*(e^B \wedge j^*\Pi_{\nabla}(\mathfrak{d}) \varphi) = q_*(e^B \wedge j^*\Pi_{\nabla}(\mathfrak{d}') \varphi),
$$
where $f\in C^{\infty}(M_{red})$ is such that $q^*f = \det(\pr_{D_2}|_{D_1})$ and $\pr_{D_2}: E|_N \Arrow D_2$ is the projection along $L|_N \cap K$.
\end{rem}

\section{Applications.}

\subsection{Generalized Calabi-Yau reduction.}

Let $E$ be a Courant algebroid over $M$ and consider its complexification $E_{\C} = E\otimes \C$. By extending $\C$-bilinearly both the metric and the Courant bracket, we can study the Lagrangian subbundles of $E_{\C}$ whose sections are closed under $\Cour{\cdot, \cdot}_{\C}$. We call such Lagrangian subspaces \textit{complex Dirac structures} on $M$.

\begin{dfn}[M. Gualtieri \cite{gualt_thesis}, N. Hitchin\cite{hitchin}]\label{gen_complex}\em
A \textit{generalized complex structure} on $M$ is a complex Dirac structure $L \subset E_{\C}$ such that
\begin{equation}\label{total_real}
L \cap \overline{L} = 0,
\end{equation}
where $\overline{L}$ is the conjugate subbundle. A general Lagrangian subbundle $L \subset E_{\C}$ such that \eqref{total_real} holds is called a \textit{generalized almost complex structure}.
\end{dfn}

A generalized almost complex structure $L \subset E_{\C}$ on $M$ can be equivalently described (see \cite{crainic, gualt_thesis}) by a bundle map $\J: E \lrrow E$ such that $\J^2=-Id$ and $\langle \J \cdot, \J \cdot \rangle = \langle \cdot, \cdot\rangle$. Under this description, $L$ is the $+i$-eigenbundle of $\J$,
\begin{equation}\label{L_complex}
L= \{ e - i \J e\,\, | \,\, e \in E\}.
\end{equation}
The integrability of $L$ is equivalent to
$$
\Cour{\J e_1, \J e_2} - \Cour{e_1, e_2} - \J(\Cour{\J e_1, e_2}  + \Cour{e_1, \J e_2}) = 0, \,\,\, \forall \, e_1, \, e_2 \in \Gamma(E).
$$

It is straightforward to extend the definition of pure spinor line bundle for the case of a generalized complex structure $L$ so as to have a complex line bundle $U_{\nabla}(L) \subset \Ext{T^*M}\otimes \C$ over $M$.

\begin{dfn}{\em
A \textit{generalized Calabi-Yau structure} on $M$ is an almost generalized complex structure $L \subset E_{\C}$ such that  $U_{\nabla}(L)$ has a global nowhere-zero section $\varphi$ such that
\begin{equation}\label{gcy_eq}
d_H\varphi=0,
\end{equation}
where $H \in \Omega^3(M)$ is the curvature of $\nabla$ and $d_H$ is the $H$-twisted differential.
}
\end{dfn}

\begin{rem}\em
Note that Equation \eqref{gcy_eq} implies that $L$ is integrable because of \eqref{spinor_integrability}.
\end{rem}

\begin{exam}[\cite{gualt_thesis, hitchin}]\em
Let $J: TM \Arrow TM$ be an almost complex structure on $M$ and $H \in \Omega^3_{cl}(M)$. Consider $\J: \T M \Arrow \T M$ given by
\begin{equation}\label{complex_complex}
\J = \left(\begin{array}{cc}
           -J & 0\\
           0 & J^*
          \end{array} \right).
\end{equation}
The corresponding almost generalized complex structure \eqref{L_complex} is
$$
L = T_{0, \, 1} \oplus \Ann{T_{0,\,1}}
$$
and its pure spinor line bundle is $\wedge^{n,0}T^*M$. $L$ defines a generalized Calabi-Yau structure if and only if the canonical line bundle $\wedge^{n,\,0} T^*M$ has a global nowhere-zero closed section $\varphi$ and $H \in \Omega^{3,0}(M)$. This implies that $J$ is integrable; moreover, both $\varphi$ and $H$ are holomorphic.
\end{exam}

The next example shows how a symplectic structure can be seen as a generalized Calabi-Yau structure.

\begin{exam}[\cite{gualt_thesis, hitchin}]\label{symp_calabi}\em
Let $(M,\omega)$ be a symplectic manifold. Define $\J: \T M \Arrow \T M$ by
$$
\J= \left(\begin{array}{cc}
           0 & \omega^{-1}_{\sharp}\\
           -\omega_{\sharp} & 0
          \end{array} \right).
$$
The corresponding Lagrangian subbundle \eqref{L_complex} is
$$
L= \{ X + i \,\omega_{\sharp} (X) \,\, | \,\, X \in \Gamma(TM)\otimes \C\};
$$
its pure spinor line bundle is generated by $e^{-i\omega} \in \Gamma(\Ext{T^*M}) \otimes \C$, which is nowhere-zero and closed. 
\end{exam}

In \cite{b_c_g}, Theorem \ref{reduced_dirac} is extended so as to encompass complex Dirac structures $L \subset E_{\C}$. More specifically, if $L|_N \cap K_{\C}$ has constant rank, then
$$
L_{red}= \left. \frac{L|_N \cap K_{\C}\p + K_{\C}}{K_{\C}} \right/G
$$
defines a complex Dirac structure on $(E_{\C})_{red}$, where $K_{\C}= K\otimes \C$ is the complexification of the isotropic subbundle \eqref{isotropic}. It is also proven in \cite{b_c_g} (see Lemma 5.1 therein) that if $L$ is the $+i$-eigenbundle of a generalized complex structure $\J : E \Arrow E$, then
\begin{equation}\label{complex_transv}
L_{red} \cap \overline{L}_{red} = 0 \Leftrightarrow \J K \cap K\p \subset K.  
\end{equation}

In what follows, fix an invariant generalized Calabi-Yau structure $L\subset E_{\C}$ on $M$ which is the $+i$-eigenbundle of $\J: E \Arrow E$ and a $(\mathcal{A}, \chi, N)$-admissible splitting $\nabla$. We now study conditions on $L$ which guarantees that $L_{red}$ is a generalized Calabi-Yau structure on $M_{red}$. A general theorem in this respect should address two questions:

\begin{itemize}
\item[(1)] Does $U_{\nabla_{red}}(L_{red})$ have a nowhere-zero global section $\varphi_{red}$?
\item[(2)] Is $\varphi_{red}$ closed under the $H_{red}$-twisted differential $d_{H_{red}}$?
\end{itemize}

With respect to the first question, we have the following proposition relating the first Chern classes of $U_{\nabla}(L)$ and $U_{\nabla_{red}}(L_{red})$.

\begin{prop}\label{chern_prop}
Suppose $L|_N \cap K_{\C}$ has constant rank. The Chern classes of $U_{\nabla}(L)$ and $U_{\nabla_{red}}(L_{red})$ are related by
\begin{equation}\label{chern_class}
c_1(U_{\nabla_{red}}(L_{red})) = c_1(U_{\nabla}(L|_N)/G) + c_1(\det(L|_N\cap K_{\C})^*/G).
\end{equation}
\end{prop}

\begin{proof}
Choose a pertubation input $D \subset E_{\C}|_N$ for $(L, K_{\C})$. Let $\{\mathcal{W}_{\alpha}\}$ be a open cover of $M_{red}$ such that $U_{\nabla}(L_D)/G$ is trivial over each $\mathcal{W}_{\alpha}$. By choosing invariant sections $\varphi_{\alpha} \in \Gamma(U_{\nabla}(L_D)|_{q^{-1}(\mathcal{W}_{\alpha})})$, the corresponding cocycle
$
g_{\alpha \beta}: q^{-1}(\mathcal{W}_{\alpha}\cap \mathcal{W}_{\beta}) \Arrow \C^*
$
satisfies 
$$
g_{\alpha\beta} = \tilde{g}_{\alpha\beta} \circ q,
$$
where $\{\tilde{g}_{\alpha \beta}\}$ is a cocycle for $U_{\nabla}(L_D)/G$. By Theorem \ref{main_transv}, one has that
$$
\varphi_{\alpha, \, red}:=q_*(j^*\varphi_{\alpha})
$$
is a nowhere-zero section of $U_{\nabla_{red}}(L_{red})$ over $\mathcal{W}_{\alpha}$. Hence, over $\mathcal{W}_{\alpha} \cap \mathcal{W}_{\beta}$,
$$
\varphi_{\alpha, \, red} = q_*(j^*\varphi_{\alpha}) = q_*(q^* \tilde{g}_{\alpha \beta}\, j^*\varphi_{\beta}) = \tilde{g}_{\alpha \beta} \,q_*(j^*\varphi_{\beta}) = \tilde{g}_{\alpha \beta} \, \varphi_{\beta,\, red}, 
$$
which proves that $\{\tilde{g}_{\alpha \beta}\}$ is also a cocycle for $U_{\nabla_{red}}(L_{red})$. For a partition of unity $\{\rho_{\alpha}\}$ subordinate to $\{\mathcal{W}_{\alpha}\}$, one has
$$
c_1(U_{\nabla_{red}}(L_{red})|_{\mathcal{W}_{\alpha}} = -\frac{1}{2\pi i} \sum_{\gamma} d(\rho_{\gamma} \, d \log \tilde g_{\gamma \alpha})
= c_1(U_{\nabla}(L_D)/G)|_{\mathcal{W}_{\alpha}}.
$$
This shows that $c_1(U_{\nabla_{red}}(L_{red}) = c_1(U_{\nabla}(L_D)/G)$. By Remark \ref{inv_D}, one has that
$$
c_1(U_{\nabla}(L_D)/G) = c_1(U_{\nabla}(L|_N)/G) + c_1(\det(D)/G).
$$
The result now follows from the fact that the map $D \ni e \mapsto \langle e,  \cdot \rangle \in (L|_N \cap K_{\C})^*$  is an $G$-equivariant isomorphism.
\end{proof}

Let us give an example illustrating the role of quotienting out by $G$ in the right hand side of \eqref{chern_class}.

\begin{exam}\label{cp1}\em
Consider $M=\C^2$ and its canonical complex structure $J: TM\Arrow TM$. Let $S^1$ act on $M$ by 
$$
e^{i\theta} \cdot (z_1, z_2) = (e^{i\theta}z_1, e^{i\theta} z_2)
$$
and lift the action to $(\T M, \Cour{\cdot, \cdot}_{H})$ with zero moment one-form and $H=0$. Consider the invariant submanifold $N= S^3$ and let $K$ be the associated isotropic subbundle of $\T M|_N$ given by
$$
K= \{u_M \,\, |\,\, u \in \mathfrak{s}^1\}|_N \oplus \Ann{N}.
$$
The generalized complex structure $\J$ \eqref{complex_complex} corresponding to $J$ is $S^1$-invariant, 
$$
U(\J)= \wedge^{2,0} T^*M
$$ 
admits a global nowhere-zero closed section given by $\varphi= dz_1\wedge dz_2$. Hence $\J$ is generalized Calabi-Yau. It also satisfies $\J K \cap K\p=0$, which implies the transversality condition $L|_N \cap K_{\C}=0$. The reduced generalized complex structure $\J_{red}$ is just the usual complex structure on $\C P^1$. In this case,  Proposition \ref{chern_prop} says that
$$
c_1(\wedge^{2,0} T^*\C^2|_{S^3}/S^1) = c_1(\wedge^{1,0} T^* \C P^1).
$$
As $\wedge^{1,0} T^* \C P^1$ has non-vanishing Chern class, it follows from Proposition \ref{chern_prop} that although $\wedge^{2,0} T^*\C^2$ has a nowhere-zero global section it does not admit nowhere-zero  $S^1$-invariant sections.
\end{exam}

The question related to the existence of closed sections is more subtle. We give first steps towards answering (2).

\begin{thm}\label{transv_calabi}
If $\J K \cap K\p =0$ and $U_{\nabla}(L)$ has an invariant nowhere-zero $d_H$-closed section $\varphi$, then $L_{red}$ is a generalized Calabi-Yau structure.
\end{thm}

\begin{proof}
First note that $L|_N \cap K_{\C} = \{ e - i\J e\,\, | \,\, e \in K\cap \J K\}$. Hence, $\J K \cap K\p =0$ implies that $L|_N \cap K_{\C}$. By Theorem \ref{main_transv}, 
$$
\varphi_{red} = q_*(j^*\varphi)
$$
is a nowhere-zero global section of $U_{\nabla_{red}}(L_{red})$. Now, as $d$ commutes with $q_*$ (see \cite{b_t}), one has
$$
d_{H_{red}}\varphi_{red} = q_*(j^* d\varphi) - H_{red} \wedge q_*(j^*\varphi).
$$
Finally, using that $q^*H_{red} = j^*H$, one has that $H_{red} \wedge q_*(j^*\varphi) = q_*(j^*(H\wedge \varphi))$. Hence,
$$
d_{H_{red}}\varphi_{red} = q_*(j^* d_H \varphi) = 0.
$$
The fact that $L_{red}\cap \overline{L}_{red} =0$ follows from \eqref{complex_transv}.
\end{proof}

\begin{exam}\em
In the same setting of Example \ref{cp1}, note that by considering the invariant open set
$
\mathcal{W}_1 = \{(z_1, z_2) \in \C^2 \,\, | \,\, z_1\neq 0\},
$
instead of the whole $\C^2$, we now have that 
$U(\J)=\wedge^{2,0} T^* \mathcal{W}_1$ has a nowhere-zero $S^1$-invariant closed section given by
$$
\varphi_1 = \frac{1}{z_1^2} \, dz_1 \wedge dz_2.
$$
The corresponding nowhere-zero closed section of $U(\J_{red})$ over $q(\mathcal{W}_1)$ is 
$$
q_*(j^*\varphi_1) = dz,
$$
where $z:q(\mathcal{W}_1) \Arrow \C$ is the coordinate
$$
z: [z_1, z_2] \mapsto \frac{z_2}{z_1}.
$$  
\end{exam}

In the case $\J K \cap K\p \neq 0$, we have to put restrictions on the intersection $L|_N \cap K_{\C}$.  

\begin{thm}\label{gen_calabi_thm}
Suppose $L|_N \cap K_{\C}$ has  constant rank and that  $U_{\nabla}(L_D)$ has an invariant nowhere-zero $d_H$-closed section $\varphi$. If $\J K \cap K\p \subset K$ and $p: L|_N \cap K_{\C} \Arrow \ker(dq) \otimes \C$ is an isomorphism , then $L_{red}$ is a generalized Calabi-Yau structure on $M_{red}$.
\end{thm}

Before proving Theorem \ref{gen_calabi_thm}, let us show how it recovers Nitta's result \cite{nitta}.

\begin{exam}[\cite{nitta}]\label{nitta_exam}\em
Consider the isotropic $G$-lifted action on $(\T M, \Cour{\cdot, \cdot})$ of Example \ref{model} with zero moment one-form and zero 3-form. Consider an invariant generalized Calabi-Yau structure $\J : \T M \Arrow \T M$  and suppose there exists an equivariant map $\mu: M \Arrow \frakg^*$
(with respect to the co-adjoint action) such that
\begin{equation}\label{nitta_cond}
\J u_M =  d\mu^u, \,\, \forall \, u\in \frakg.
\end{equation}
Assume $0$ is a regular value of $\mu$ and take the invariant submanifold $N=\mu^{-1}(0)$. Now, \eqref{nitta_cond} implies that
\begin{equation}\label{l_cap_k}
L|_{\mu^{-1}(0)} \cap K_{\C} = \{ (u_M + i\,v_M, \, d\mu^v - i\, d\mu^{u}) \,\, | \,\, u, \, v \in \frakg\},
\end{equation}
where $L$ is the $+i$-eigenbundle of $\J$. If one assumes, as usual, that $G$ acts freely on $\mu^{-1}(0)$, then the restriction of $\pr_{TM}$ to $L|_{\mu^{-1}(0)}\cap K_{\C}$ is an isomorphism over $\ker(dq)\otimes\C$. Moreover, as \eqref{nitta_cond} implies that $\J K = K$, it follows that $\J K \cap K\p= K$. This is exactly the setting in which Y. Nitta \cite{nitta} performed the reduction of generalized Calabi-Yau structures. We refer to \cite{nitta} to see applications of his result to Duistermaat-Heckmann type formulas for generalized Calabi-Yau structures.
\end{exam}

For the proof of Theorem \ref{gen_calabi_thm} we shall need a Lemma concerning the push-forward map. 

\begin{lem}\label{push_connection}
Let $\theta \in \Omega^1(N, \frakg)$ be a connection 1-form on $N$ and let $\ker(\theta):=\{X \in TN \,\, | \,\, i_X \theta =0\}$. For a basis $\{u^1, \cdots, u^r\}$ of\, $\frakg$, 
consider the decomposition $\theta = \sum_{i=1}^r \theta_i u^i$, where $\theta_i \in \Omega^1(N)$. One has that $\theta_{[1,r]} = \theta_1 \wedge \cdots \wedge \theta_r$ is an invariant section of $\Ann{\ker(\theta)}$ and
$$
q_*(\theta_{[1,r]}) = \int_G \nu,
$$
where $\nu \in \det(T^*G)$ is the left-invariant volume form on $G$ such that $i_{u^r \wedge \cdots \wedge u^1} \nu (e) = 1$.
\end{lem}

\begin{proof}
The fact that $\theta_{[1,r]}$ is invariant follows directly from the invariance of $\theta$,
$$
\varphi^*_g \theta = \sum_{i=1}^r \theta_i \,Ad_{g^{-1}}(u^i), \,\, g\in G,
$$
and the fact that $\det(Ad_{g^{-1}})=1$ (as $G$ is compact and connected). As for the second statement, let $\mathcal{U}$ be an open set of $M_{red}$ such that $\pi^{-1}(\mathcal{U})\cong \mathcal{U} \times G$. Consider  the basis $\{\xi_1, \dots, \xi_r\}$ of $\frakg^*$ dual to $\{u^1, \dots, u^r\}$. Define $\alpha_i:=\theta_i - \pr_G^*\, \xi_i^L$; it is straightforward to check that $i_{u_N} \alpha_i=0$, for $i=1, \dots, r$ and $u\in \frakg$. Thus, by expanding
$$
\theta_1 \wedge \cdots \wedge \theta_r = \left(\alpha_1 + \pr_G^*\,\xi_1^L\right)\wedge\cdots\wedge\left(\alpha_r + \pr_G^*\,\xi_r^L\right),
$$
we see that
$$
\theta_1 \wedge \cdots \wedge \theta_r = \pr_G^* \,\nu + \text{forms of type} \,(II)
$$
The result now follows from \eqref{push_def}.
\end{proof}

\begin{proof}[Proof of Theorem \ref{gen_calabi_thm}.]
First note that the hypothesis that $p: L|_N \cap K_{\C}\Arrow \ker(dq)\otimes \C$ is an isomorphism implies that $L|_N \cap K_{\C}$ has constant (complex) rank equal to $\dim(G)$. Let $L_{red} \subset E_{red}\otimes \C$ be the reduced complex Dirac structure. To find the pure spinor line bundle of $L_{red}$, we must use the pertubative method of \S 3.2.2 as $L|_N \cap K_{\C} \neq 0$. For this, let $\theta \in \Omega^1(N, \frakg)$ be a connection 1-form and choose an invariant complement $QN \subset TM|_N$ for $TN$. Define
$$
D= p^*(\Ann{\ker(\theta) \oplus QN}) \otimes \C.
$$
It is an invariant isotropic subbundle of $E_{\C}|_N$. We claim that $D$ is a pertubation input for $(L, K_{\C})$ (i.e. $(L|_N \cap K_{\C})\oplus D\p = E_{\C}|_N$). Indeed, let $e \in E_{\C}$. Then
$$
\begin{array}{rl}
\langle e, p^*\xi \rangle = 0 , \, \forall \, p^*\xi \in D & \Leftrightarrow \xi(p(e)) = 0, \, \forall \, \xi \in \Ann{\ker(\theta) \oplus QN} \otimes \C\vspace{2pt}\\
&  \Leftrightarrow p(e) \in (\ker(\theta)\oplus QN)\otimes \C.
\end{array}
$$
As $p: L|_N \cap K_{\C} \Arrow \ker(dq)\otimes \C$ is an isomorphism and $\ker(dq) \cap (\ker(\theta)\oplus QN) =0$, it follows that $(L|_N \cap K_{\C})\cap D\p = 0$. Our claim now follows from dimension count as $\dim_{\C}(L|_N \cap K_{\C}) = \dim_{\C}(D) = \dim(G)$. Write $\theta=\sum_{i=1}^r \theta_i u^i$, where $\{u_1, \dots, u_r\}$ is a basis of $\frakg$ and extend $\theta_i \in \Omega(N)$ to $\tilde{\theta}_i \in \Gamma(\Ext{T^*M}|_N)$ by  $\tilde{\theta}_i|_{QN} \equiv 0$. By Lemma \ref{push_connection}, $\mathfrak{d} := p^*(\tilde{\theta}_1 \wedge \cdots \wedge \tilde{\theta}_r) \in \Gamma(\det(D))$ is an invariant section. So, Theorem \ref{main_geral} gives that 
$$
\varphi_{red} = q_*(j^* \Pi_{\nabla}(\mathfrak{d}) \varphi) = q_*(\theta_1 \wedge \cdots \wedge \theta_r \wedge j^*\varphi)
$$
is a nowhere-zero global section of $U_{\nabla_{red}}(L_{red})$. We claim that $j^*\varphi$ is a basic form. Indeed, it is invariant by hypothesis. Now, let $u \in \frakg$ and consider $k_u$, the unique section of $L|_N \cap K_{\C}$ such that $p(k_u) = u_M$. As $\nabla$ is $(\mathcal{A}, \chi, N)$-admissible, we have that
$$
k_u = \nabla u_M + p^*\eta, 
$$
where $\eta \in \Ann{TN}\otimes \C$ (see Remark \ref{K_nabla}). Hence, as $k_u \in L|_N$, 
$$
0=j^* \Pi_{\nabla}(k_u) \varphi = i_{u_N} j^*\varphi + j^*\eta \wedge j^* \varphi = i_{u_N}j^*\varphi,
$$
as we claimed. Let $\varphi_0 \in \Omega(M_{red})$ be such that $q^*\varphi_0 = j^*\varphi$. By Lemma \ref{push_connection}, we have that
$$
\varphi_{red} = q_*(\theta_1 \wedge \cdots \wedge \theta_r \wedge q^*\varphi_0) = q_*(\theta_1 \wedge \cdots \wedge \theta_r)\, \varphi_0 =\left(\int_G \nu\right) \varphi_0,
$$
where $\nu$ is the left-invariant volume form on $G$ such that $i_{u_r \wedge \cdots \wedge u_1}\nu(e)=1$. Finally,
$$
\begin{array}{rl}
(\int_G \nu)^{-1} q^*(d\varphi_{red} - H_{red} \wedge \varphi_{red}) = & dq^*\varphi_0 - q^*H_{red} \wedge q^*\varphi_0 \vspace{2pt}\\
=&   dj^* \varphi - j^*H \wedge j^*\varphi\vspace{2pt}\\
=&   j^*(d\varphi - H\wedge \varphi) = 0,
\end{array}
$$
which implies that $d\varphi_{red} - H_{red} \wedge \varphi_{red} =0$ (we have used that $q^*H_{red}= j^*H$, see Remark \ref{reduced_curvature}). Hence, $U_{\nabla_{red}}(L_{red})$ has a nowhere-zero global $d_{H_{red}}$-closed section $\varphi_{red}$. The fact that $L_{red}$ is a generalized complex structure follows from \eqref{complex_transv}. This completes the proof.
\end{proof}

\subsection{T-duality.}

In this subsection, we explore the striking similarity of formula \eqref{main_spin} with the T-duality map introduced in \cite{bou_ev}. More precisely, we show how formula \eqref{main_spin} gives an alternative explanation of the fact that the T-duality map preserves pure spinors. This will be based on recent results obtained by G. Cavalcanti and M. Gualtieri \cite{c_g, cavalcanti} relating T-duality with reduction procedure of \cite{b_c_g}.

Let $\pi_1: P_1 \Arrow N$ be a principal circle bundle with an invariant closed integral 3-form $H_1 \in \Omega^3(P_1)$ and a connection 1-form $\theta_1 \in \Omega^1(P_1)$ . We have identified $\mathfrak{s}^1 \cong \R$ in such a way that $\pi_1\vphantom{}_* \theta_1=1$ (see Lemma \ref{push_connection}). Define
$$
c_2 := \pi_1\vphantom{}_* H \in \Omega^2(N)
$$
and let $c_1 \in \Omega^2(N)$ be the curvature of $P_1$ (i.e. $\pi^*c_1 = d\theta_1$). There exists $h\in \Omega^3(N)$ such that 
\begin{equation}\label{H1}
H_1= \pi_1^* c_2 \wedge \theta_1 + \pi_1^*h.
\end{equation}
By general properties of the push-forward map, $[c_2] \in H^2(N, \mathbb{Z})$. Therefore, there exists a principal circle bundle $\pi_2: P_2\Arrow N$ with a connection 1-form $\theta_2 \in \Omega^1(P_2)$ whose curvature is $c_2$.
Define
$$
H_2 = \pi_2^* c_1 \wedge \theta_2 + \pi_2^*h \in \Omega^3(P_2).
$$
$(P_2, \theta_2, H_2)$ is called the \textit{T-dual space} corresponding to $(P_1, \theta_1, H_1)$. We refer to \cite{bou_ev} for the physical interpretation of T-duality and examples of T-dual spaces (see also \cite{c_g,rosenberg})

Given T-dual spaces $(P_1,\theta_1,  H_1)$ and $(P_2, \theta_2, H_2)$,  define the \textit{correspondence space} to be the fiber product $M= P_1 \times_N P_2$ of $P_1$ and $P_2$. The natural projections $q_1: M \Arrow P_1$ and $q_2: M \Arrow P_2$, which make the diagram below commutative,
$$
\xy
(0,10)*++^{P_1}="P_1"; (10,0)*+^{N}="N"; (10,20)*+^{M}="M"; (20,10)*++^{P_2}="P_2";
{\ar@{->}_{q_1}"M"; "P_1"}; 
{\ar@{->}^{q_2}"M"; "P_2"};
{\ar@{->}_{\pi_1}"P_1"; "N"};
{\ar@{->}^{\pi_2}"P_2"; "N"};
\endxy
$$
also give $M$ the structure of a principal circle bundle over $P_1$ and $P_2$ respectively. Let $\Omega_{S^1}(P_i)$ be the space of invariant differential forms on $P_i$ and consider the $H_i$-twisted differential $d_{H_i}$, for $i=1,2$. As $H_1$ is invariant, then $d_{H_1}$ restricts to $\Omega(P_1)^{S^1}$ turning it into a differential complex (similarly for $d_{H_2}$ and $\Omega(P_2)^{S^1}$).

\begin{thm}\cite{bou_ev}
Let $B= q_1^*\theta_1 \wedge q_2^*\theta_2 \in \Omega^2(M)$. The map $\tau: (\Omega(P_1)^{S^1}, d_{H_1}) \Arrow (\Omega(P_2)^{S^1}, d_{H_2})$ defined by
\begin{equation}\label{tau_def}
\tau = q_2 \vphantom{}_* \circ e^{B} \circ q_1^*
\end{equation}
is an isomorphism of differential complexes.
\end{thm}

In the remainder of this paper, we give an alternative proof of the following result

\begin{prop}\cite{c_g}
The T-duality map $\tau$ \eqref{tau_def} preserves pure spinors.
\end{prop}

\begin{proof}
Consider the Courant algebroid $E=(\T M, \Cour{\cdot, \cdot}_{q^*_1 H_1})$. It carries an isotropic $S^1$-lifted action $(\mathcal{A}, \chi)$, where $\mathcal{A}$ is defined by \eqref{standard_action} (corresponding to the principal circle bundle $q_2: M \Arrow P_2$) and $\chi: \R \Arrow \Gamma(\T M)$ given by
$$
\chi (\mathbf{1}) = \mathbf{1}_M + q_2^* \theta_2.
$$
Take $N = M$ as an invariant submanifold and note that $q_1^*\theta_1$ defines a connection on $N$ such that the corresponding 2-form defined by \eqref{B_def} is exactly $B=q_1^*\theta_1\wedge q_2^*\theta_2$. It is now straightforward to check that the splitting $\nabla_{red}$ \eqref{nabla_red} for $E_{red} = (K\p/K)/G$ has curvature $H_2$ and identify $E_{red}$ with $(\T P_2, \Cour{\cdot, \cdot}_{H_2})$. Let $\varphi \in \Omega_{S^1}(M)$ be a pure spinor.  One has that $q^*_1 \varphi \in \Omega_{S^1}(M)$ is a pure spinor (see Proposition 1.5 in \cite{a_b_m} for a proof) such that 
$$
\mathcal{N}(q^*_1\varphi) = \{(Y, dq_1^*\eta) \in \T M\,\, | \,\, dq_1(Y) + \eta \in \mathcal{N}(\varphi)\}.
$$
It follows that $\mathcal{N}(q^*_1 \varphi) \cap K=0$. Hence, Theorem \ref{main_transv} gives that $\tau(\varphi) = q_2\vphantom{}_*(e^B \wedge q^* \varphi)$ is a section of the pure spinor line bundle $U_{\nabla_{red}}(\mathcal{N}(q_1^*\varphi)_{red})$. This shows completes the prove.
\end{proof}

\end{document}